\documentclass[12pt]{amsart}
\usepackage{amsmath,amssymb,amsthm}
\usepackage{amssymb}
\newtheorem{theorem}{Theorem}[section]
\newtheorem{lema}[theorem]{Lemma}
\newtheorem{cor}[theorem]{Corollary}
\newtheorem{prop}[theorem]{Proposition}
\theoremstyle{definition}
\newtheorem{defi}[theorem]{Definition}
\newtheorem{remark}[theorem]{Remark}

\usepackage{multicol}
\usepackage{graphicx}
\usepackage{xcolor}
\usepackage{enumitem}
\newcommand{\subscript}[2]{$#1 _ #2$}

\usepackage{tikz}
\usepackage{tikz-cd}
\newcommand{\arrowIn}{\tikz \draw[-stealth] (-1pt,0) -- (1pt,0);}
\usepackage{mathtools}
\usepackage[all]{xy}

\usepackage[pdftex,%
bookmarks=true,
breaklinks=true,
bookmarksnumbered = true,
colorlinks= true,
urlcolor= green,
anchorcolor = yellow,
citecolor=blue,
]{hyperref}

\numberwithin{equation}{section}

\begin{document}

\title[The BNS invariants of the braid groups and pure braid groups of some surfaces]{The BNS invariants of the braid groups and pure braid groups of some surfaces}
\author{Carolina de Miranda e Pereiro}
\address{UFES - Universidade Federal do Espírito Santo}
\email{carolina.pereiro@ufes.br}
\author{Wagner Sgobbi}
\address{UFSCar - Universidade Federal de São Carlos}
\email{wagnersgobbi@ufscar.br}

\date{\today}


\begin{abstract}
We compute and explicitly describe the Bieri-Neumann-Strebel invariants $\Sigma^1$ for the full and pure braid groups of the sphere $\mathbb{S}^2$, the real projective plane $\mathbb{R}P^2$ and specially the torus $\mathbb{T}$ and the Klein bottle $\mathbb{K}$. In order to do this for $M=\mathbb T$ or $M=\mathbb K$, and $n \geq 2$, we use the $n^{th}$-configuration space of $M$ to show that the action by homeomorphisms of the group $Out(P_n(M))$ on the character sphere $S(P_n(M))$ contains certain permutation of coordinates, under which $\Sigma^1(P_n(\mathbb T))^c$ and $\Sigma^1(P_n(\mathbb K))^c$ are invariant. Furthermore, $\Sigma^1(P_n(\mathbb T))^c$ and $\Sigma^1(P_n(\mathbb{S}^2))^c$ (the latter with $n \geq 5$) are finite unions of pairwise disjoint circles, and $\Sigma^1(P_n(\mathbb K))^c$ is finite. This last fact implies that there is a normal finite index subgroup $H \leq Aut(P_n(\mathbb K))$ such that the Reidemeister number $R(\varphi)$ is infinite for every $\varphi \in H$.

\end{abstract}

\maketitle

\section{Introduction}

The Bieri-Neumann-Strebel invariant $\Sigma^1(G)$ of a finitely generated group $G$ was first defined in \cite{BNS}, based on a previously defined invariant for metabelian groups \cite{BS}. Given such group $G$ and a finitely generated $G$-operator group $A$, the authors of \cite{BNS} associate to $A$ a subset $\Sigma_A=\Sigma_A(G)$ of the character sphere $S(G)$. 
Many general properties of $\Sigma_A$ are shown in \cite{BNS}, including its openness in $S(G)$ and a characterization of the finitely generated normal subgroups of $G$ containing the commutator subgroup $G'$ \cite[Theorem A4.1]{Strebel}. Nowadays, the invariant $\Sigma^1=\Sigma_{G'}$ is well known as an important object in Geometric Group Theory which has many connections to other topics, including closed $1$-forms on $3$-manifolds, group actions on $\mathbb{R}$-trees \cite{Brown}, and the Thurston norm and fibering of manifolds over $S^1$ (see e.g. Theorem E of \cite{BNS} and \cite{FK}).

Despite of its importance, the BNS-invariant is in general hard to be effectively computed and there are only a few classes of groups for which $\Sigma^1$ is already known (we refer to \cite{KL} and its references for a list of such families). One can try to avoid this problem by using the Cayley graph definition of $\Sigma^1(G)$ in \cite{Strebel}, which is the basic literature for this geometric approach. Then, the well known Geometric Criterion for $\Sigma^1$ (Theorem \ref{geometriccrit}) provides us with a more elementary 
approach to $\Sigma^1$ than the original definition of \cite{BNS}. This is especially useful when dealing with groups $G$ which are mostly known by a finite presentation, such as the full and pure braid groups of surfaces, which are the object of study of this paper. Another tool that can be useful for the computation of $\Sigma^1(G)$ is the study of the action by homeomorphisms $Out(G) \curvearrowright S(G)$. Since $\Sigma^1(G)\subset S(G)$ is invariant under this action, knowing how certain automorphisms act on the sphere can give us information about the shape of $\Sigma^1(G)$. This will be specially used on the cases $M=\mathbb T$ and $M=\mathbb K$, with $n \geq 2$ (see Section \ref{sec:aut}).

First, we recall the definitions and some properties about the braid groups. The braid groups of the $2$-disc $\mathbb D^{2}$, also called the Artin braid groups, were defined by E. Artin~\cite{Ar0} as follows. Let $Q=\left\{q_{1},\ldots,q_{n}\right\}$ be a set of $n$ distinct points of $\mathbb D^{2}$. A $n$-geometric braid is a collection of paths $f=(f_{1},\ldots,f_{n})$ in $\mathbb D^{2}$ such that  $f_{i}(t)\neq f_{j}(t)$ for all $t\in[0,1]$ and $0\leq i <j\leq n$, $f_{i}(0)=q_{i}$ and $f_{i}(1)=q_{\tau(i)}$ for some $\tau\in S_{n}$, for all $1\leq i \leq n$, where $S_{n}$ is the symmetric group on $n$ letters. The path $f_{i}$ is called $i$-th string.  One possible way of visualize the geometric braids in $\mathbb D^{2}$ is to take a cylinder $\mathbb D^{2}\times [0,1]$, and place the $n$ points $q_{1},\ldots, q_{n}$ on the top and on the bottom face of the cylinder, then join each point $q_{i}$ on the top with the $q_{\tau(i)}$ point on the bottom following the paths $f_{i}$, notice that the strings of a braid flow monotonically downwards, as in Figure~\ref{fig:artin}.

The Artin braid group with $n$-strings, denoted by $B_{n}$, is the set of $n$-geometric braid, up isotopy, equipped with the concatenation operation. For $n\geq 2$, $B_{n}$ has a well known finite presentation, with the elements $\sigma_{1},\ldots,\sigma_{n-1}$ as generators, subject to the Artin relations:
\begin{align*}
\left\{\begin{array}{lr}\sigma_{i}\sigma_{i+1}\sigma_{i}=\sigma_{i+1}\sigma_{i}\sigma_{i+1}&\mbox{for all}\,1\leq i\leq n-2,\\
\sigma_{i}\sigma_{j}=\sigma_{j}\sigma_{i}&\mbox{if}\ |i-j|\geq 2\ \mbox{and}\ 1\leq i,j\leq n-1.\end{array}\right.    
\end{align*}
The elements $\sigma_{i}$, also called as Artin generators, and its inverses, can be
visualized in Figure~\ref{fig:artin}. There is a natural correspondence that associates each braid $f=(f_{1},\ldots,f_{n})$ to the permutation $\tau$ of $S_{n}$. In terms of the Artin generators, it is the projection $\pi:B_{n}\longrightarrow S_{n}$, $\sigma_{i}\mapsto (i\ \ \ i+1)$. The Artin pure braid group with $n$-strings $P_{n}$ is
isomorphic to the kernel of $\pi$.

\begin{figure}[h!]
\centering
\begin{tikzpicture}
\draw (0,0) ellipse (2 and 0.3);
\draw (-2,0) -- (-2,-2);
\draw (2,-2) -- (2,0);  
\draw (-2,-2) arc (180:360:2 and 0.3);
\draw [dashed] (-2,-2) arc (180:360:2 and -0.3);

\draw[white,line width=6pt] (0.35,0).. controls (0.35,-1) and (-0.35,-1).. (-0.35,-2);
\draw[blue] (0.35,0).. controls (0.35,-1) and (-0.35,-1).. (-0.35,-2);
\draw[white,line width=6pt] (-0.35,0).. controls (-0.35,-1) and (0.35,-1).. (0.35,-2);
\draw[blue] (-0.35,0).. controls (-0.35,-1) and (0.35,-1).. (0.35,-2);


\foreach \j in {-0.9,0.9,-1.6,1.6}
{\draw[white,line width=4pt] (\j,0)--(\j,-2);};
\foreach \j in {-0.9,0.9,-1.6,1.6}
{\draw (\j,0)--(\j,-2);};

\draw[white,line width=6pt] (-2,-2) arc (180:360:2 and 0.3);
\draw (-2,-2) arc (180:360:2 and 0.3);
\draw[white,line width=6pt] (0,0) ellipse (2 and 0.3);
\draw (0,0) ellipse (2 and 0.3);
\foreach \j in {0.9,-0.9,-1.6,1.6, 0.35, -0.35}
{\filldraw[black] (\j,0) circle (1pt);};
\foreach \j in {-0.9,0.9,-1.6,1.6, 0.35,-0.35}
{\filldraw[black] (\j,-2) circle (1pt);};
\draw (-2,0) -- (-2,-2);
\draw (2,-2) -- (2,0);  

\node at (1.25,-1) {\small$\cdots$};
\node at (-1.25,-1) {\small$\cdots$};
\node at (0,-2.8) {$\sigma_{i}$};
\node at (-0.35,0.5) {\scriptsize$i$};
\node at (0.35,0.5) {\scriptsize$i+1$};
\node at (-1.6,0.5) {\scriptsize$1$};
\node at (1.6,0.5) {\scriptsize$n$};

\begin{scope}[shift={(6.5,0)}]
\draw (0,0) ellipse (2 and 0.3);
\draw (-2,0) -- (-2,-2);
\draw (2,-2) -- (2,0);  
\draw (-2,-2) arc (180:360:2 and 0.3);
\draw [dashed] (-2,-2) arc (180:360:2 and -0.3);

\draw[white,line width=6pt] (-0.35,0).. controls (-0.35,-1) and (0.35,-1).. (0.35,-2);
\draw[blue] (-0.35,0).. controls (-0.35,-1) and (0.35,-1).. (0.35,-2);
\draw[white,line width=6pt] (0.35,0).. controls (0.35,-1) and (-0.35,-1).. (-0.35,-2);
\draw[blue] (0.35,0).. controls (0.35,-1) and (-0.35,-1).. (-0.35,-2);


\foreach \j in {-0.9,0.9,-1.6,1.6}
{\draw[white,line width=4pt] (\j,0)--(\j,-2);};
\foreach \j in {-0.9,0.9,-1.6,1.6}
{\draw (\j,0)--(\j,-2);};

\draw[white,line width=6pt] (-2,-2) arc (180:360:2 and 0.3);
\draw (-2,-2) arc (180:360:2 and 0.3);
\draw[white,line width=6pt] (0,0) ellipse (2 and 0.3);
\draw (0,0) ellipse (2 and 0.3);
\foreach \j in {0.9,-0.9,-1.6,1.6, 0.35, -0.35}
{\filldraw[black] (\j,0) circle (1pt);};
\foreach \j in {-0.9,0.9,-1.6,1.6, 0.35,-0.35}
{\filldraw[black] (\j,-2) circle (1pt);};
\draw (-2,0) -- (-2,-2);
\draw (2,-2) -- (2,0);  

\node at (1.25,-1) {\small$\cdots$};
\node at (-1.25,-1) {\small$\cdots$};
\node at (0,-2.8) {$\sigma^{-1}_{i}$};
\node at (-0.35,0.5) {\scriptsize$i$};
\node at (0.35,0.5) {\scriptsize$i+1$};
\node at (-1.6,0.5) {\scriptsize$1$};
\node at (1.6,0.5) {\scriptsize$n$};
\end{scope}

\end{tikzpicture}
\caption{The Artin generators $\sigma_{i}$ and $\sigma^{-1}_{i}$.}\label{fig:artin}
\end{figure}
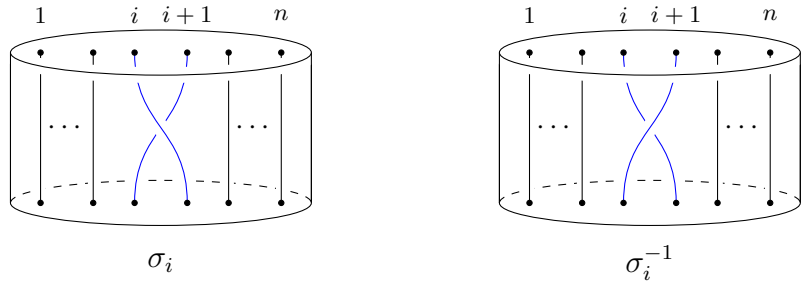

Further, R.~H.~Fox and L.~Neuwirth~\cite{FoN} generalised the Artin (pure) braid groups for surfaces using configuration space as follows. Let $M$ be a compact, connected surface, obtained by removing a finite number (possibly zero)  of points in the interior of the surface. If $n\geq 1$, the $n$-configuration space of $M$, denoted by $F_{n}(M)$, is defined by:
\begin{align*}
    F_{n}(M)=\left\{(x_{1},\ldots,x_{n})\,:\,x_{i}\in M,\, \mbox{and}\,x_{i}\neq x_{j}\,\mbox{if}\,i\neq j,\,i,j=1,\ldots,n\right\}.
\end{align*}
Fix $Q=(q_{1},\ldots,q_{n})$ a point of $F_{n}(M)$. The $n$-pure braid group $P_{n}(M)$ of $M$ is defined by $P_{n}(M)=\pi_{1}(F_{n}(M), Q)$. The symmetric group $S_{n}$ acts freely on $F_{n}(M)$ by permuting coordinates, and the $n$-braid group $B_{n}(M)$ of $M$ is deﬁned by $B_{n}(M)=\pi_{1}(F_{n}(M)/S_{n}, \bar{Q})$, where $\bar{Q}$ denote the class of $Q$ in $F_{n}(M)/S_{n}$.

The theory of (surface) braid groups has been gaining importance over the years. Several algebraic properties of these groups are known, such as their finite presentations, study of the center, study of torsion elements, the lower central and derived series, among others (see~\cite{B1,BGG, FVB,GG,GG3,GG5,GM,GP,PR,P} and also the survey~\cite{GJP}, and its references). Furthermore, these groups have proven to be a valuable tool when dealing with problems of low-dimensional algebraic topology, such as knot theory, fixed points and coincidence theory, Borsuk-Ulam property, multivalued functions and other topics~\cite{GG1,GG2,LP}.

Regarding the BNS invariants and braid groups, N.~Koban, J.~McCammond and J.~Meier~\cite{KMM} were able to compute $\Sigma^{1}(P_{n})$, for the Artin pure braid groups, and showed that its complement in the sphere $S(G)$ is a finite union of certain circles (more details are given on Section 5). Also, two years later, Zaremsky \cite{Zaremsky} investigated the higher BNS invariants (\emph{a.k.a.} the BNSR-invariants) $\Sigma^m(P_n),\ m \geq 1$, using Morse theory. Motivated by these papers and by the recent studies of braid groups of surfaces, we investigate the BNS invariants for the braid groups and pure braid groups of some surfaces, namely the sphere, the projective plane, the torus and the Klein bottle. As far as we know, there have been no computations on the BNS invariants for braid groups and pure braid groups of other surfaces besides the disc $\mathbb{D}^2$ in the literature. We obtain the following.


\begin{theorem}\label{th:PnS2}Let $n \geq 3$. If $n=3$, the BNS-invariant $\Sigma^1(P_{n+1}(\mathbb S^{2}))$ is empty. If $n \geq 4$, then $\Sigma^1(P_{n+1}(\mathbb S^{2}))$ is the complement of the union of some $P_3$-circles $\widetilde{C}_{i,j,k}$, for $1 \leq i < j < k \leq n$, and $P_4$-circles $\widetilde{C}_{i,j,k,l}$, for $1 \leq i < j < k < l \leq n$ (Definition \ref{def:circles}) in its character sphere. There are exactly $\binom{n}{3}+\binom{n}{4}$ such circles, which are pairwise disjoint.
\end{theorem}

If $M$ is either the torus $\mathbb T$ or the Klein bottle $\mathbb K$, the fact that $B_{n}(M)$ has a non-trivial center, as well as the knowledge of a set of generators for $Z(B_{n}(M))$, was essential in our next results. If $M$ is an orientable (resp. non-orientable) closed surface with genus $g\geq 2$ (resp. $g\geq 3$), the same techniques cannot be used, because $Z(B_{n}(M))$ is trivial~\cite{PR}.

\begin{theorem}\label{th:Bn} Let $M$ be either the torus $\mathbb T$ or the Klein bottle $\mathbb K$. Then \[\Sigma^{1}(B_{n}(M))=S(B_{n}(M))\simeq \left\{\begin{array}{ll}\mathbb S^{1},&\,\,\mbox{if}\,\,M=\mathbb T;\\
\mathbb S^{0},&\,\,\mbox{if}\,\,M=\mathbb K.\end{array}\right.\]
\end{theorem}

The next result is the main theorem of our paper. It is interesting to observe that the complement of $\Sigma^{1}$ for the pure braid groups of the torus resembles the results for the pure braid groups of both the disc and the sphere, being the union of pairwise disjoint circles. On the other hand, in the case of the Klein bottle we obtain a finite set, which differs significantly from the previous ones.

\begin{theorem}\label{BNS-Klein} Let $M$ be either the torus $\mathbb T$ or the Klein bottle $\mathbb K$ and $n\geq2$. The complement $\Sigma^1(P_n(M))^c$ inside the sphere $S(P_{n}(M))\cong \mathbb S^{2n-1}$ if $M=\mathbb T$ (respectively $S(P_{n}(M))\cong \mathbb S^{n-1}$ if $M=\mathbb K$)  is given by
\begin{align*}
\Sigma^{1}(P_{n}(\mathbb T))^c&=\{[\chi_{i,j,p,q}]\ |\ 1 \leq i,j \leq n, i\neq j, (p,q)\neq 0\},
\end{align*}
where  $\chi_{i,j,p,q}(a_i)=\chi_{i,j,p,q}(a^{-1}_j)=p$, $\chi_{i,j,p,q}(b_i)=\chi_{i,j,p,q}(b^{-1}_j)=q$ and $\chi_{i,j,p,q}(a_{k})=\chi_{i,j,p,q}(b_{k})=0$ for $k \neq i,j$. This is a union of $\binom{n}{2}$ pairwise disjoint circles. And, 
\begin{align*}
\Sigma^{1}(P_{n}(\mathbb K))^c&=\{[\chi_{i,j}]\ |\ 1 \leq i,j \leq n, i\neq j\},
\end{align*}
where $\chi_{i,j}(b_i)=\chi_{i,j}(b^{-1}_j)=1$ and $\chi_{i,j}(b_{k})=0$ for $k \neq i,j$. This is a union of $2\binom{n}{2}$ points.
\end{theorem}

Theorem \ref{BNS-Klein} imply the following algebraic result (more details on Section \ref{sec:appl}).

\begin{cor}\label{cor:twistedpnklein}
For any $n \geq 2$, there exists a normal subgroup $H \leq Aut(P_n(\mathbb K))$ of finite index $|Aut(P_n(\mathbb K)):H|\leq \left(2 \binom{n}{2}\right)!$ such that the Reidemeister number $R(\varphi)$ is infinite for every $\varphi \in H$.
\end{cor}

This manuscript is organised as follows. On Section~\ref{sec:sigma} we briefly recall the BNS invariant and present our main tool to compute $\Sigma^1$. On Section~\ref{sec:braid} we present the surface braid groups, focusing on the cases were the surface is the torus and the Klein bottle. We derive some presentations and relations which are useful to prove Theorem~\ref{BNS-Klein}. On Section~\ref{sec:aut} we find certain useful automorphisms of $P_n(M)$, which give us valuable information on $\Sigma^1(P_n(M))$ for $M=\mathbb T$ and $M=\mathbb K$ and enable its computation. Section~\ref{sec:computation} is then devoted to the computation of $\Sigma^1$ for the groups $B_n(\mathbb{R}P^2)$, $P_n(\mathbb{R}P^2)$, $B_n(\mathbb{S}^2)$ and to the proofs of Theorems~\ref{th:PnS2}, \ref{th:Bn} and especially \ref{BNS-Klein}. We finish on Section~\ref{sec:appl}, by showing some applications regarding commutator subgroups (Corollary \ref{cor:g'naoef.g.}) and twisted conjugacy (Corollary \ref{cor:twistedpnklein}) for the groups in question.

\section*{Acknowledgements}

We would like to thank Prof. Peter Wong (Bates College - USA) for pointing out this research to us, Prof. Daniel Vendrúscolo (UFSCar - Brazil) for the help with some braid visualizations and Prof. Daciberg Gonçalves (IME-USP - Brazil) for the help with the braid groups of the sphere. The second author would like to thank Fundação de Amparo à Pesquisa do Estado de São Paulo
(FAPESP) for the financial support during the research through process 2022/07198-0.

\section{Preliminaries on $\Sigma^1$}\label{sec:sigma}

We begin by introducing the invariant $\Sigma^1(G)$ of a finitely generated group $G$. There are many equivalent definitions in the literature (see \cite{BNS} and \cite{Strebel}), but we choose the following Cayley graph approach of \cite{Strebel}. Remember that the character sphere of a finitely generated group $G$ is the quotient space
\[
S(G)=(Hom(G,\mathbb{R})-\{0\})/\sim \ = \{[\chi]\ |\ \chi \in Hom(G,\mathbb{R})-\{0\}\}
\] of nonzero homomorphisms $\chi:G \to \mathbb{R}$ (characters), where $\chi \sim \chi' \Leftrightarrow r\chi=\chi'$ for some $r>0$.

Let $\gamma_n(G)$, $n \geq 1$, be the terms of the lower central series of $G$, that is, $\gamma_1(G)=G$ and $\gamma_{n+1}(G)=[\gamma_n(G),G]$, for any $n\geq 1$. It is well known that if the free rank of the abelianized group $G^{Ab}=G/\gamma_2(G)$ is $r\geq 1$ with generators $x_1,...,x_r$, then $S(G) \simeq \mathbb{S}^{r-1}$, with homeomorphism
\begin{align*}
  \mathfrak h:S(G)&\longrightarrow \mathbb{S}^{r-1} \\
  [\chi]&\longmapsto \frac{(\chi(x_1),\ldots,\chi(x_r))}{\Vert (\chi(x_1),\ldots,\chi(x_r)) \Vert}.\\
\end{align*}
Each automorphism $\varphi \in Aut(G)$ gives rise to a sphere homeomorphism $\varphi^*:S(G) \to S(G)$, $[\chi] \mapsto [\chi \circ \varphi]$, and there is a natural left-action by homeomorphisms $Aut(G) \curvearrowright S(G)$ with $\varphi \cdot [\chi]=[\chi \circ \varphi^{-1}]$. We say that a subset $A \subset S(G)$ is invariant under $\varphi$ if $\varphi^*(A) \subset A$. We say that $A$ is invariant under automorphisms if $A$ is invariant under $\varphi$ for every $\varphi \in Aut(G)$. Since an inner automorphism acts trivially on $S(G)$, one can consider the action above as an action $Out(G) \curvearrowright S(G)$. This action is particularly important for the study of twisted conjugacy on $G$: it is well known since \cite{DaciDess} (see also \cite{S}) that the existence of an invariant finite set of rational points inside an open hemisphere of $S(G)$ guarantees property $R_\infty$ for $G$.

Now, we recall the definition of $\Sigma^1$ as in \cite{Strebel}. Let $G$ be a finitely generated group and fix a finite generating set $X \subset G$. Denote by $\Gamma=\Gamma(G,X)$ the Cayley graph of $G$ with respect to $X$. The first $\Sigma$-invariant (or BNS invariant) of $G$ is
\[
\Sigma^1(G)=\{[\chi] \in S(G)\ |\ \Gamma_\chi\ \text{is connected}\},
\] where $\Gamma_\chi$ is the subgraph of $\Gamma$ whose vertices are the elements $g \in G$ such that $\chi(g) \geq 0$ and whose edges are those of $\Gamma$ which connect two such vertices.

It is well known that isomorphic finitely generated groups possess homeomorphic BNS invariants \cite[Section B1.2a]{Strebel}. The complementar of $\Sigma^1(G)$ in the sphere $S(G)$ will be denoted by $\Sigma^1(G)^c$ and the center of $G$ by $Z(G)$. It is also known ~\cite[Proposition B1.5]{Strebel} that both $\Sigma^1(G)$ and $\Sigma^1(G)^c$ are invariant under automorphisms of $G$. Assume from now on that $G$ is finitely generated. In this work, we will make use of the following standard properties of the $\Sigma^1$-invariant, together with many other important ones that can be found in \cite{S,Strebel}.

\begin{prop}[\cite{Strebel}, Proposition A2.4]\label{prop:centro} If a point $[\chi]\in S(G)$ is such that $\chi(Z(G))\neq \{0\}$, then $[\chi]\in\Sigma^{1}(G)$.
\end{prop}

\begin{theorem}[\cite{Strebel}, Proposition A2.7]\label{sigmacart}
If $G=G_1 \times G_2$ is the direct product of two finitely generated groups, then
\[
\Sigma^1(G)^c={\pi_1}^*(\Sigma^1(G_1)^c)\cup {\pi_2}^*(\Sigma^1(G_2)^c),
\]where ${\pi_i}^*:S(G_i)\longrightarrow S(G)$ with $[\chi] \longmapsto [\chi \circ \pi_i]$ are induced from the projections $\pi_i:G \to G_i$.
\end{theorem}

For the next property, we use the following notation: a path in the Cayley graph $\Gamma=\Gamma(G,X)$ of $G$ is denoted by $p=(g,z_1\cdots z_n)$, with $z_i \in Z=X^{\pm}$. The path $p$ starts at $g$, walks through the edge $(g,z_1)$ until the vertex $gz_1$, walks through $(gz_1,z_2)$ until $gz_1z_2$ and so on, until it reaches its terminus $gz_1\cdots z_n$. Given $\chi \in Hom(G,\mathbb{R})$, the evaluation function $\nu_\chi$ is given by
\[
\nu_\chi(p)=\min\{\chi(g),\chi(gz_1),\ldots,\chi(gz_1\cdots z_n)\}.
\]

With this notation, it is clear that the path $p$ is inside $\Gamma_\chi$ if, and only if, $\nu_\chi(p) \geq 0$, in which case we say that $p$ is a \emph{$\chi$-nonnegative} path. Similarly, we say $p$ is \emph{$\chi$-positive} if $\nu_\chi(p) > 0$. The following theorem is called the ``Geometric Criterion for $\Sigma^1$''.

\begin{theorem}[\cite{Strebel}, Theorem A3.1] \label{geometriccrit}
Let $G$ be a finitely generated group with finite generating set $X$ and denote $Z=X^\pm$. Let $[\chi] \in S(G)$ and choose $t \in Z$ such that $\chi(t)>0$. Then the following statements are equivalent:
\begin{itemize}
\item[1)]$\Gamma_\chi$ is connected (or $[\chi] \in \Sigma^1(G)$);\\
\item[2)]For every $z \in Z$, there exists a path $p_z$ from $t$ to $zt$ in $\Gamma$ such that $\nu_\chi(p_z)>\nu_\chi((1,z))$.
\end{itemize}
\end{theorem}

\begin{remark}\label{paths}
We will use Theorem \ref{geometriccrit} above in Section~\ref{sec:computation}, in the proof of Theorem \ref{BNS-Klein}, to guarantee that certain points belong to $\Sigma^1(P_n(M))$, for $M$ the torus or the Klein bottle. To find such paths $p_z$, one needs certain knowledge of the Cayley graph of the group $G=P_n(M)$ in question; in particular, since $p_z$ must go from a vertex $t$ to $zt$, it is useful for us to know relations of the form $tg=zt$ for certain $g \in G$ or, equivalently, of the form $t^{-1}zt=g$. Relations of this type are explored in the next section to help us in Section~\ref{sec:computation}.
\end{remark}

\section{Preliminaries on the braid groups of surfaces}\label{sec:braid}

Let $M$ be a closed surface. In this section we present some general properties about $B_{n}(M)$ and $P_{n}(M)$. Moreover, we will examine in greater detail the particular case in which $M$ is either the torus $\mathbb T$ or the Klein bottle $\mathbb K$, thereby obtaining specific relationships for these groups. In order to obtain information about them, we will provide a description of their generators, as well as some presentations of these groups. If $M=\mathbb S^{2}$ it is possible to generate $B_{n}(\mathbb S^{2})$ (resp. $P_{n}(\mathbb S^{2})$) with the same generators of the Artin braid group $B_{n}$ (resp. Artin pure braid group $P_{n}$)~\cite{FVB,GG}, but it has some additional relations. If $M$ is a closed orientable (resp. non-orientable) surface of genus $g\geq 1$, \emph{i.e.} $M\neq \mathbb S^{2}$, we can visualize $M$ as a polygon whose edges are identiﬁed as indicated in Figure~\ref{fig:superficie}.
\begin{figure}[h!]
	\begin{tikzpicture}[scale=0.45]
\draw[dotted] (3,0)--(0,0);
\draw (5.5,2)--(3,0)node[sloped, pos=0.5, allow upside down]{\arrowIn};;
\draw (6.2,4.5)--(5.5,2)node[sloped, pos=0.5, allow upside down]{\arrowIn};;
\draw (6.2,4.5)--(5.5,7) node[sloped, pos=0.5, allow upside down]{\arrowIn};;
\draw (5.5,7)--(3,9) node[sloped, pos=0.5, allow upside down]{\arrowIn};;

\draw[dotted] (3,9)--(0,9);
\draw (0,9)--(-2.5,7)node[sloped, pos=0.5, allow upside down]{\arrowIn};;
\draw (-2.5,7)--(-3.2,4.5)node[sloped, pos=0.5, allow upside down]{\arrowIn};;
\draw (-2.5,2)--(-3.2,4.5)node[sloped, pos=0.5, allow upside down]{\arrowIn};;
\draw (0,0)--(-2.5,2) node[sloped, pos=0.5, allow upside down]{\arrowIn};;

\node at (6.5,5.8) {\tiny $\epsilon_{1}$};
\node at (4.65,8.3) {\tiny $\epsilon_{2}$};
\node at (-1.85,8.45) {\tiny $\epsilon_{2g-1}$};
\node at (-3.6,5.7) {\tiny $\epsilon_{2g}$};

\node at (-3.35,3.2) {\tiny $\epsilon_{1}$};
\node at (-1.5,0.6) {\tiny $\epsilon_{2}$};
\node at (5.35,0.7) {\tiny $\epsilon_{2g-1}$};
\node at (6.65,3.2) {\tiny $\epsilon_{2g}$};

\draw[dotted] (13+3,0)--(13+0,0);
\draw (13+5.5,2)--(13+3,0)node[sloped, pos=0.5, allow upside down]{\arrowIn};;
\draw[red] (13+5.5,2)--(13+6.2,4.5)node[sloped, pos=0.5, allow upside down]{\arrowIn};;
\draw (13+6.2,4.5)--(13+5.5,7) node[sloped, pos=0.5, allow upside down]{\arrowIn};;
\draw (13+5.5,7)--(13+3,9) node[sloped, pos=0.5, allow upside down]{\arrowIn};;

\draw[dotted] (13+3,9)--(13+0,9);
\draw (13+0,9)--(13-2.5,7)node[sloped, pos=0.5, allow upside down]{\arrowIn};;
\draw[red] (13-2.5,7)--(13-3.2,4.5)node[sloped, pos=0.5, allow upside down]{\arrowIn};;
\draw (13-2.5,2)--(13-3.2,4.5)node[sloped, pos=0.5, allow upside down]{\arrowIn};;
\draw (13+0,0)--(13-2.5,2) node[sloped, pos=0.5, allow upside down]{\arrowIn};;

\node at (13+6.5,5.8) {\tiny $\epsilon_{1}$};
\node at (13+4.65,8.3) {\tiny $\epsilon_{2}$};
\node at (13-1.85,8.45) {\tiny $\epsilon_{g-1}$};
\node at (13-3.6,5.7) {\tiny $\epsilon_{g}$};

\node at (13-3.35,3.2) {\tiny $\epsilon_{1}$};
\node at (13-1.5,0.6) {\tiny $\epsilon_{2}$};
\node at (13+5.35,0.7) {\tiny $\epsilon_{g-1}$};
\node at (13+6.65,3.2) {\tiny $\epsilon_{g}$};
 \end{tikzpicture}
\caption{Orientable and non-orientable surface}\label{fig:superficie}
\end{figure}
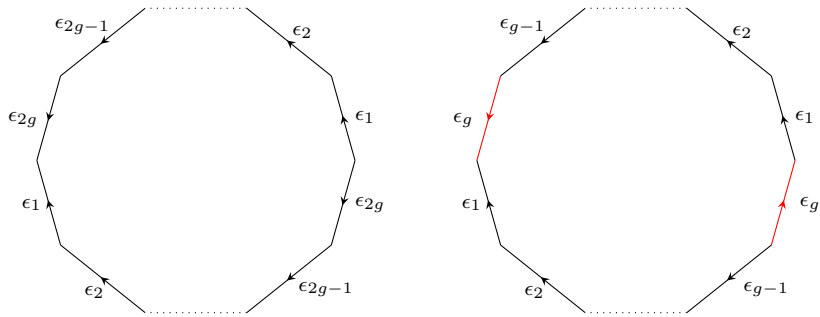
 One possible way of visualize the geometric braids in $M$ is to take a cylinder $M\times [0,1]$ and represent them similarly to the braids in the disc, but with the additional property that a string could ``cross a wall" of the cylinder, from one side to another. To simplify the drawing, we can also imagine that we are looking to the cylinder from above, and we use arrows to indicate the direction of the string. These two ways of visualization are illustrated in Figure~\ref{fig:dois}, when $M$ is the torus. 
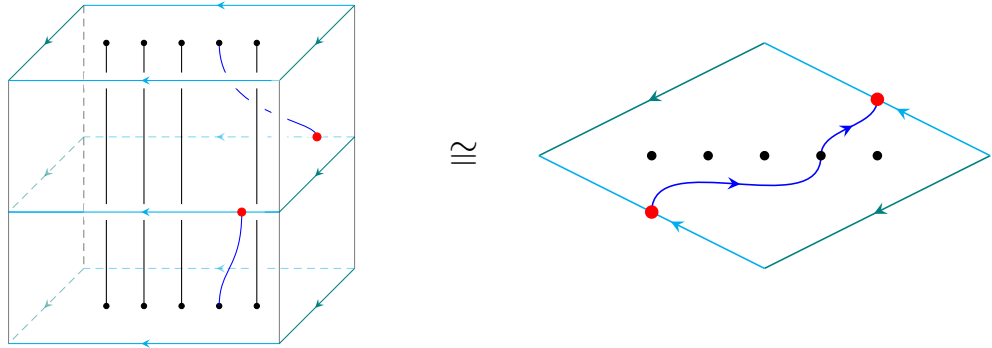
\begin{figure}[!h]
\centering
\begin{tikzpicture}[scale=1.0]
  \begin{scope}[shift={(-7,0)}]
\draw[cyan](1.3,-0.5)--(-2.3,-0.5) node[sloped, pos=0.5, allow upside down]{\arrowIn};;
\draw[cyan](1.3,-4)--(-2.3,-4) node[sloped, pos=0.5, allow upside down]{\arrowIn};;
\draw[cyan](2.3,0.5)--(-1.3,0.5) node[sloped, pos=0.5, allow upside down]{\arrowIn};;
\draw[cyan,densely dashed, opacity=.5](2.3,-3)--(-1.3,-3) node[sloped, pos=0.5, allow upside down]{\arrowIn};;

\draw[cyan] (-2.3,-2.25)--(1.3,-2.25) node[sloped, pos=0.5, allow upside down]{\arrowIn};;
\draw[cyan,densely dashed, opacity=.5] (2.3,-1.25)  -- (-1.3,-1.25) node[sloped, pos=0.5, allow upside down]{\arrowIn};;
\draw[teal,densely dashed, opacity=.5] (-1.3,-1.25)--(-2.3,-2.25) node[sloped, pos=0.5, allow upside down]{\arrowIn};;
\draw[teal] (2.3,-1.25)--(1.3,-2.25) node[sloped, pos=0.5, allow upside down]{\arrowIn};;

\draw[opacity=.5]   (-2.3,-0.5)--(-2.3,-4);
\draw[opacity=.5] (2.3,0.5)--(2.3,-3);
\draw[densely dashed, opacity=.5] (-1.3,0.5)--(-1.3,-3);

\draw[teal](2.3,0.5)--(1.3,-0.5) node[sloped, pos=0.5, allow upside down]{\arrowIn};;
\draw[teal](-1.3,0.5)--(-2.3,-0.5) node[sloped, pos=0.5, allow upside down]{\arrowIn};;
\draw[teal,densely dashed, opacity=.5] (-1.3,-3)--(-2.3,-4) node[sloped, pos=0.5, allow upside down]{\arrowIn};;
\draw[teal](2.3,-3)--(1.3,-4) node[sloped, pos=0.5, allow upside down]{\arrowIn};;

\draw[white,line width=6pt] (0.5,0) .. controls (0.5,-1) and (1.8,-1) .. (1.8,-1.25);
\draw[blue] (0.5,0) .. controls (0.5,-1) and (1.8,-1) .. (1.8,-1.25);


\draw[white,line width=6pt] (1,0)--(1,-3.5);
\foreach \j in {0,-0.5,-1,1}
{\draw (\j,0)--(\j,-3.5);};



\foreach \j in {0,0.5,1,-0.5,-1}
{\filldraw[black] (\j,0) circle (1pt);};
\foreach \j in {0,0.5,1,-0.5,-1}
{\filldraw[black] (\j,-3.5) circle (1pt);};

\draw[white,line width=6pt](1.3,-0.5)--(-1.3,-0.5);
\draw[cyan](1.3,-0.5)--(-2.3,-0.5) node[sloped, pos=0.5, allow upside down]{\arrowIn};;
\draw[white,line width=6pt] (1.3,-0.5)--(1.3,-2);
\draw[opacity=.5] (1.3,-0.5)--(1.3,-4);
\draw[white,line width=6pt] (-1.3,-2.25)--(1.1,-2.25);
\draw[cyan] (1.3,-2.25) -- (-2.3,-2.25) node[sloped, pos=0.5, allow upside down]{\arrowIn};;

\draw[blue] (0.8,-2.25) .. controls (0.8,-3) and (0.5,-3) .. (0.5,-3.5);
\filldraw[red] (0.8,-2.25) circle (1.5pt);
\filldraw[red] (1.8,-1.25) circle (1.5pt);

\end{scope}
\node at (-3.25,-1.5) {\large $\cong$};

\begin{scope}[shift={(0.5,0)}]
 \scalebox{1.5}{
\draw[blue] (0.5,-6+5).. controls (0.5 ,-5.75+5 ) and (1 ,-5.75+5 ) .. (1 ,-5.5+5 )  node[sloped, pos=0.5, allow upside down]{\arrowIn};;

\draw[blue] (-1 ,-6.5+5 ) .. controls (-1 ,-5.9+5 ) and (0.5 ,-6.6+5 ) .. (0.5 ,-6 +5)  node[sloped, pos=0.5, allow upside down]{\arrowIn};;



{\draw[teal] (0 ,-1+1)-- (-2 ,-2+1)  node[sloped, pos=0.5, allow upside down]{\arrowIn};;}
{\draw[teal]  (2 ,-2+1) -- (0 ,-3+1) node[sloped, pos=0.5, allow upside down]{\arrowIn};;}
{\draw[cyan] (2 ,-2+1) -- (0 ,-1+1) node[sloped, pos=0.4, allow upside down]{\arrowIn};;}
{\draw[cyan] (0 ,-3+1) -- (-2 ,-2+1) node[sloped, pos=0.4, allow upside down]{\arrowIn};;}

\foreach \j in {0,0.5,1,-0.5,-1}
{\filldraw[black] (\j ,-2+1) circle (1pt);};


\filldraw[red] (1 ,-0.5 ) circle (1.5pt);
\filldraw[red] (-1 ,-1.5 ) circle (1.5pt);
}
\end{scope}

\end{tikzpicture}
\caption{A braid in the torus}\label{fig:dois}
\end{figure}
There are several presentations for the (pure) braid groups of $M$ in the literature, with several differences in notation, as well as in generators and relations depending on the way they are chosen to represent the surface. We have chosen the visualization above, which resembles what was done in~\cite{GM,GP,P}. Some common properties of theses groups are the following. For $1\leq i \leq n$ and $1\leq r \leq 2g$, if $M$ orientable (resp. $1\leq r \leq g$, if $M$ is non-orientable), we can consider $\rho_{i,r}$ the braid in $P_{n}(M)$ such that the $i$-th string is the only non-trivial one, which crosses the edge $\epsilon_{i}$. For $1 \leq j<k \leq n$, consider $C_{j,k}$ the braid whose $k$-{th} string is the only non-trivial one, which encircles all the basepoints between the $j$-{th} and $k$-{th} points, counterclockwise. These braids can be visualized in Figure~\ref{fig:geradores} and the elements $\left\{\rho_{i,r},C_{j,k}\right\}$ generate $P_{n}(M)$. A set of generators of $B_{n}(M)$ is the union of the generators of $P_{n}(M)$ with the Artin generators of $B_{n}$.

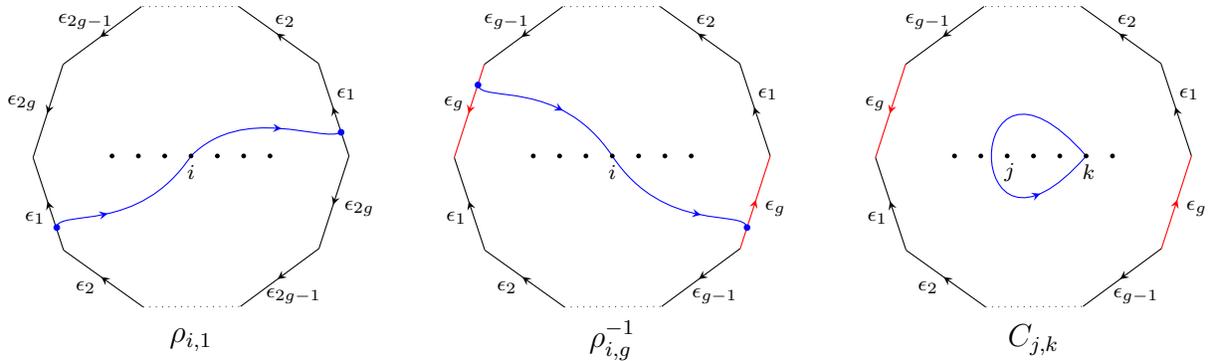
\begin{figure}[h]
\hfill
\begin{tikzpicture}[scale=0.7]
\draw (-8+3,0)--(-8+2.42,1.76) node[sloped, pos=0.5, allow upside down]{\arrowIn};;
\draw (-8+2.42,1.76)--(-8+0.91,2.84) node[sloped, pos=0.5, allow upside down]{\arrowIn};;
\draw[opacity=.75, dotted] (-8+0.91,2.84)--(-8-0.94,2.84);
\draw (-8-0.94,2.84)--(-8-2.43,1.74) node[sloped, pos=0.5, allow upside down]{\arrowIn};;
\draw (-8-2.43,1.74)--(-8-3,-0.03) node[sloped, pos=0.5, allow upside down]{\arrowIn};;
\draw (-8-2.42,-1.79)--(-8-3,-0.03) node[sloped, pos=0.5, allow upside down]{\arrowIn};;
\draw (-8-0.91,-2.87)--(-8-2.42,-1.79) node[sloped, pos=0.5, allow upside down]{\arrowIn};;
\draw[opacity=.75, dotted] (-8-0.91,-2.87)--(-8+0.94,-2.86);
\draw (-8+2.43,-1.76)--(-8+0.94,-2.86) node[sloped, pos=0.5, allow upside down]{\arrowIn};;
\draw (-8+3,0)--(-8+2.43,-1.76) node[sloped, pos=0.5, allow upside down]{\arrowIn};;

\node at (-8+2.96,1.18) {\tiny$\epsilon_{1}$};
\node at (-8+-2.96,-1.17) {\tiny$\epsilon_{1}$};
\node at (-8+1.79,2.57) {\tiny$\epsilon_{2}$};
\node at (-8-2,-2.5) {\tiny$\epsilon_{2}$};
\node at (-8-2,2.55) {\tiny$\epsilon_{2g-1}$};
\node at (-8+1.95,-2.65) {\tiny$\epsilon_{2g-1}$};
\node at (-8-3.2,1) {\tiny$\epsilon_{2g}$};
\node at (-8+3.2,-1) {\tiny$\epsilon_{2g}$};
\node at (-8+0,-0.3) {\tiny$i$};


\draw[blue] (-8-0,0) .. controls (-8+1,1) and (-8+2.5,0.25)..  (-8+2.85,0.45) node[sloped, pos=0.5, allow upside down]{\arrowIn};;

\draw[blue] (-8-2.56,-1.36).. controls (-8-2.5,-1) and (-8-1,-1.5) .. (-8+0,0)  node[sloped, pos=0.5, allow upside down]{\arrowIn};;


{\filldraw[blue] (-8-2.55,-1.36) circle (1.5pt);}
{\filldraw[blue] (-8+2.85,0.45) circle (1.5pt);}

{\filldraw[black] (-8+0,0) circle (1pt);}
{\filldraw[black] (-8+0.5,0) circle (1pt);}
{\filldraw[black] (-8+1,0) circle (1pt);}
{\filldraw[black] (-8-0.5,0) circle (1pt);}
{\filldraw[black] (-8-1,0) circle (1pt);}
{\filldraw[black] (-8+1.5,0) circle (1pt);}
{\filldraw[black] (-8-1.5,0) circle (1pt);}

\draw (3,0)--(2.42,1.76) node[sloped, pos=0.5, allow upside down]{\arrowIn};;
\draw (2.42,1.76)--(0.91,2.84) node[sloped, pos=0.5, allow upside down]{\arrowIn};;
\draw[opacity=.75, dotted] (0.91,2.84)--(-0.94,2.84);
\draw (-0.94,2.84)--(-2.43,1.74) node[sloped, pos=0.5, allow upside down]{\arrowIn};;
\draw[red] (-2.43,1.74)--(-3,-0.03) node[sloped, pos=0.5, allow upside down]{\arrowIn};;
\draw (-2.42,-1.79)--(-3,-0.03) node[sloped, pos=0.5, allow upside down]{\arrowIn};;
\draw (-0.91,-2.87)--(-2.42,-1.79) node[sloped, pos=0.5, allow upside down]{\arrowIn};;
\draw[opacity=.75, dotted] (-0.91,-2.87)--(0.94,-2.86);
\draw (2.43,-1.76)--(0.94,-2.86) node[sloped, pos=0.5, allow upside down]{\arrowIn};;
\draw[red] (2.43,-1.76)--(3,0) node[sloped, pos=0.5, allow upside down]{\arrowIn};;

\node at (2.96,1.18) {\tiny$\epsilon_{1}$};
\node at (-2.96,-1.17) {\tiny$\epsilon_{1}$};
\node at (1.79,2.57) {\tiny$\epsilon_{2}$};
\node at (-2,-2.5) {\tiny$\epsilon_{2}$};
\node at (-2,2.55) {\tiny$\epsilon_{g-1}$};
\node at (1.95,-2.65) {\tiny$\epsilon_{g-1}$};
\node at (-3,1) {\tiny$\epsilon_{g}$};
\node at (3.1,-1) {\tiny$\epsilon_{g}$};
\node at (0,-0.3) {\tiny$i$};



 \draw[blue] (-2.55,1.35).. controls (-2.5,1) and (-1,1.5)..  (0,0) node[sloped, pos=0.7, allow upside down]{\arrowIn};;
 \draw[blue] (0,0).. controls (1,-1.5) and  (2.5,-1) .. (2.56,-1.36) node[sloped, pos=0.5, allow upside down]{\arrowIn};;

{\filldraw[blue] (-2.55,1.35) circle (1.5pt);}
{\filldraw[blue] (2.56,-1.36) circle (1.5pt);}

{\filldraw[black] (0,0) circle (1pt);}
{\filldraw[black] (0.5,0) circle (1pt);}
{\filldraw[black] (1,0) circle (1pt);}
{\filldraw[black] (-0.5,0) circle (1pt);}
{\filldraw[black] (-1,0) circle (1pt);}
{\filldraw[black] (1.5,0) circle (1pt);}
{\filldraw[black] (-1.5,0) circle (1pt);}
\draw (+8+3,0)--(+8+2.42,1.76) node[sloped, pos=0.5, allow upside down]{\arrowIn};;
\draw (+8+2.42,1.76)--(+8+0.91,2.84) node[sloped, pos=0.5, allow upside down]{\arrowIn};;
\draw[opacity=.75, dotted] (8+0.91,2.84)--(8-0.94,2.84);
\draw (+8-0.94,2.84)--(+8-2.43,1.74) node[sloped, pos=0.5, allow upside down]{\arrowIn};;
\draw[red] (+8-2.43,1.74)--(+8-3,-0.03) node[sloped, pos=0.5, allow upside down]{\arrowIn};;
\draw (+8-2.42,-1.79)--(+8-3,-0.03) node[sloped, pos=0.5, allow upside down]{\arrowIn};;
\draw (+8-0.91,-2.87)--(+8-2.42,-1.79) node[sloped, pos=0.5, allow upside down]{\arrowIn};;
\draw[opacity=.75, dotted] (+8-0.91,-2.87)--(+8+0.94,-2.86);
\draw (+8+2.43,-1.76)--(+8+0.94,-2.86) node[sloped, pos=0.5, allow upside down]{\arrowIn};;
\draw[red] (+8+2.43,-1.76)--(+8+3,0) node[sloped, pos=0.5, allow upside down]{\arrowIn};;

\node at (+8+2.96,1.18) {\tiny$\epsilon_{1}$};
\node at (+8+-2.96,-1.17) {\tiny$\epsilon_{1}$};
\node at (+8+1.79,2.57) {\tiny$\epsilon_{2}$};
\node at (+8-2,-2.5) {\tiny$\epsilon_{2}$};
\node at (+8-2,2.55) {\tiny$\epsilon_{g-1}$};
\node at (+8+1.95,-2.65) {\tiny$\epsilon_{g-1}$};
\node at (+8-3,1) {\tiny$\epsilon_{g}$};
\node at (+8+3.1,-1) {\tiny$\epsilon_{g}$};
\node at (+8-0.45,-0.3) {\tiny$j$};
\node at (+8+1.05,-0.3) {\tiny$k$};


\draw[blue] (9,0) .. controls (6.6,2.75) and (6.6,-2.75)..  (9,0) node[sloped, pos=0.85, allow upside down]{\arrowIn};;

{\filldraw[black] (8+0,0) circle (1pt);}
{\filldraw[black] (8+0.5,0) circle (1pt);}
{\filldraw[black] (8+1,0) circle (1pt);}
{\filldraw[black] (8-0.5,0) circle (1pt);}
{\filldraw[black] (8-1,0) circle (1pt);}
{\filldraw[black] (8+1.5,0) circle (1pt);}
{\filldraw[black] (8-1.5,0) circle (1pt);}

\node at (-8,-3.5) {$\rho_{i,1}$};
\node at (0,-3.5) {$\rho^{-1}_{i,g}$};
\node at (8,-3.5) {$C_{j,k}$};

\end{tikzpicture}
\hspace*{\fill}
\caption{Generators of $P_{n}(M)$}\label{fig:geradores}
\end{figure}

In this paper we will make use of the following presentation for $P_{n}(M)$, if $M$ is either the torus or the Klein bottle, given by~\cite{GP}. By adapting the notation above with the notation of~\cite{GP}, we have $a_{i}=\rho_{i,1}$ and $b_{i}=\rho^{-1}_{i,2}$

\begin{theorem}[\cite{GP}, Theorem 2.1]\label{puras}
Let $n\geq1$, and let $M$ be the torus $\mathbb{T}$ or the Klein bottle $\mathbb{K}$. The following constitutes a presentation of the pure braid group $P_n(M)$ of $M$:

\noindent
generators: $\left\{a_i,\, b_i,\,i=1,\ldots,n\right\}\cup\left\{C_{i,j},\,1\leq i < j\leq n\right\}$.

\noindent
relations:
\begin{enumerate}
\item\label{it:puras1} $a_ia_j=a_ja_i$, 		 $(1\leq i<j\leq n)$
	
\item\label{it:puras2} $a^{-1}_{i}b_ja_i=b_ja_jC^{-1}_{i,j}C_{i+1,j}a^{-1}_{j}$, 	 $(1\leq i<j\leq n)$
	
\item\label{it:puras3} $a^{-1}_{i}C_{j,k}a_i=\left\{\begin{array}{l} C_{j,k},\,\,  (1\leq i<j<k\leq n)\,\, \mbox{or}\,\,(1\leq j<k<i\leq n)\\
	a_kC^{-1}_{i+1,k}C_{i,k}a^{-1}_{k}C_{j,k}C^{-1}_{i,k}C_{i+1,k},\,\,  (1\leq j\leq i<k\leq n)\end{array}\right.$
	
\item\label{it:puras4} $C^{-1}_{i,l}C_{j,k}C_{i,l}=\left\{\begin{array}{l} C_{j,k},\,\,  (1\leq i<l<j<k\leq n)\,\, \mbox{or}\,\,(1\leq j\leq i<l<k\leq n)\\
	C_{i,k}C^{-1}_{l+1,k}C_{l,k}C^{-1}_{i,k}C_{j,k}C^{-1}_{l,k}C_{l+1,k},\,\,  (1\leq i< j\leq l<k\leq n)\end{array}\right.$
	
\item\label{it:puras5} $\left\{ \begin{array}{lll}
\prod^{n}_{j=i+1}C^{-1}_{i,j}C_{i+1,j}=a_{i}b_{i}C_{1,i}a^{-1}_{i}b^{-1}_{i},& 		 (1\leq i\leq n), &\text{if}\,\, M=\mathbb{T}\\
 \prod^{n}_{j=i+1}C_{i,j}C^{-1}_{i+1,j}=b_{i}C_{1,i}a^{-1}_{i}b^{-1}_{i}a^{-1}_{i},& 		 (1\leq i\leq n), &\text{if}\,\, M=\mathbb{K}\end{array}\right.$
	
\item\label{it:puras6} $\left\{\begin{array}{ll} b_{j}b_{i}=b_{i}b_{j},\,\, (1\leq i<j\leq n), &\text{if}\,\, M=\mathbb{T}\\
b_{j}b_{i}=b_{i}b_{j}C_{i,j}C^{-1}_{i+1,j},\,\, 		 (1\leq i<j \leq n), &\text{if}\,\, M=\mathbb{K}\end{array}\right.$
	
\item\label{it:puras7} $\left\{\begin{array}{ll} b^{-1}_{i}a_jb_i=a_jb_jC_{i,j}C^{-1}_{i+1,j}b^{-1}_{j},\,\, (1\leq i<j\leq n), &\text{if}\,\, M=\mathbb{T}\\
	b^{-1}_{i}a_jb_i=a_jb_j(C_{i,j}C^{-1}_{i+1,j})^{-1}b^{-1}_{j},\,\, (1\leq i<j\leq n), &\text{if}\,\, M=\mathbb{K} \end{array}\right.$
	
\item\label{it:puras8} $\left\{\begin{array}{ll}{b^{-1}_{i}C_{j,k}b_i=\left\{\begin{array}{l} C_{j,k},\,\,  (1\leq i<j<k\leq n)\,\, \mbox{or}\,\,(1\leq j<k<i\leq n)\\
	C_{i+1,k}C^{-1}_{i,k}C_{j,k}b_kC_{i,k}C^{-1}_{i+1,k}b^{-1}_{k},\,\,  (1\leq j\leq i<k\leq n)\end{array}\right.} & \!\!\!\text{if}\,\, M=\mathbb{T}\\
	{b^{-1}_{i}C_{j,k}b_i=\left\{\begin{array}{l} C_{j,k},\,\,  (1\leq i<j<k\leq n)\,\, \mbox{or}\,\,(1\leq j<k<i\leq n)\\
	C_{i+1,k}C^{-1}_{i,k}C_{j,k}b_k(C_{i,k}C^{-1}_{i+1,k})^{-1}b^{-1}_{k},\,\,  (1\leq j\leq i<k\leq n)\end{array}\right.} \!\!\!&\text{if}\,\, M=\mathbb{K}. \end{array}\right.$
	
	\end{enumerate}

\end{theorem}

In the following pages, we obtain some particular relations in $P_{n}(M)$, for $M$ the torus $\mathbb{T}$ or the Klein bottle $\mathbb{K}$. These relations are important in their own right, but they will also play an essential role in the computation of $\Sigma^1(P_n(M))$ on Theorem \ref{BNS-Klein}.

\begin{prop}\label{prop:add}

Let $M$ be the torus $\mathbb T$ or the Klein bottle $\mathbb K$. The following relations are valid in $P_{n}(M)$.
\begin{enumerate}[label=(\subscript{S}{{\arabic*}})]
\item\label{it:add1} $a_{i}b_{j}a^{-1}_{i}=b_{j}C_{i,j}C^{-1}_{i+1,j}$, $(1\leq i <j \leq n)$;

\item\label{it:add2} $b_{i}a_{j}b^{-1}_{i}=a_{j}C_{i+1,j}C^{-1}_{i,j}$, $(1\leq i<j\leq n)$;

\item \label{it:add3} $a_{i}C_{j,k}a^{-1}_{i}=C_{i+1,k}C^{-1}_{i,k}C_{j,k}a^{-1}_{k}C_{i,k}C^{-1}_{i+1,k}a_{k}$, $(1\leq j\leq i<k\leq n)$

\item\label{it:add4} $b_{i}C_{j,k}b^{-1}_{i}=\left\{\begin{array}{lr}b^{-1}_{k}C^{-1}_{i+1,k}C_{i,k}b_{k}C_{j,k}C^{-1}_{i,k}C_{i+1,k},\,\,(1\leq j\leq i<k\leq n) &\mbox{if}\,\,M=\mathbb T\\ b^{-1}_{k}(C^{-1}_{i+1,k}C_{i,k})^{-1}b_{k}C_{j,k}C^{-1}_{i,k}C_{i+1,k},\,\,(1\leq j\leq i<k\leq n) &\mbox{if}\,\,M=\mathbb K.\end{array}  \right.$

\item\label{it:add5} $b_{i}b_{j}b^{-1}_{i}=\left\{\begin{array}{lr}b_{j},\,\,(1\leq i<j\leq n) &\mbox{if}\,\,M=\mathbb T,\\
C^{-1}_{i+1,j}C_{i,j}b_{j},\,\,(1\leq i<j\leq n) &\mbox{if}\,\,M=\mathbb K;\end{array}  \right.$

\end{enumerate}
\end{prop}

\begin{proof}
To prove item~\ref{it:add1}, we use relations~(\ref{it:puras2}) and~(\ref{it:puras3}) from Theorem~\ref{puras},
\begin{align*}
b_{j}=a_{i}\Big(b_{j}a_{j}C^{-1}_{i,j}C_{i+1,j}a^{-1}_{j}\Big)a^{-1}_{i}=a_{i}\Big(b_{j}\cdot a^{-1}_{i}C_{i+1,j}C^{-1}_{i,j}a_{i}\Big)a^{-1}_{i}=a_{i}b_{j}a^{-1}_{i}\cdot C_{i+1,j}C^{-1}_{i,j}.
\end{align*}

Analogously, we can prove item~\ref{it:add2}, by using relations~(\ref{it:puras7}) and~(\ref{it:puras8}) from Theorems~\ref{puras},
\begin{align*}
a_{j}=b_{i}\Big(a_{j}b_{j}(C_{i,j}C^{-1}_{i+1,j})^{\pm1}b^{-1}_{j}\Big)b^{-1}_{i}=b_{i}\Big(a_{j}\cdot b^{-1}_{i}C^{-1}_{i+1,j}C_{i,j}b_{i}\Big)b^{-1}_{i}=b_{i}a_{j}b^{-1}_{i}\cdot C^{-1}_{i+1,j}C_{i,j}.
\end{align*}

To prove item \ref{it:add3}, we start with the case $i=j$. Using that $a_{i}$ commutes with both $C_{i+1,k}$ and $a_{k}$ by relations~(\ref{it:puras3}) and ~(\ref{it:puras1}), respectively, and using relation~(\ref{it:puras3}) for $i=j$, we obtain
\begin{align*}
C_{i,k}=a_{i}\Big(a_{k}C^{-1}_{i+1,k}C_{i,k}a^{-1}_{k}C_{i+1,k}\Big)a^{-1}_{i}= a_{k}C^{-1}_{i+1,k}\cdot a_{i} C_{i,k}a^{-1}_{i}\cdot a^{-1}_{k} C_{i+1,k},
\end{align*}
from where \ref{it:add3} follows. For $j<i$, by using the same relations above and item~\ref{it:add3} for $i=j$, we obtain
\begin{align*}
C_{j,k}=a_{i}\Big(a_{k}C^{-1}_{i+1,k}C_{i,k}a^{-1}_{k}C_{j,k}C^{-1}_{i,k}C_{i+1,k}\Big)a^{-1}_{i}= (C_{i,k}C^{-1}_{i+1,k})\cdot a_{i}C_{j,k}a^{-1}_{i}\cdot (a^{-1}_{k}C_{i+1,k}C^{-1}_{i,k}a_{k}),
\end{align*}
from where \ref{it:add3} follows. The proof of item~\ref{it:add4} is analogous to the previous one. If $M=\mathbb T$, for $i=j$, using that $b_{i}$ commutes with $C_{i+1,k}$ and $b_{k}$ by relations~(\ref{it:puras8}) and~(\ref{it:puras6}),  we obtain
\begin{align*}
C_{i,k}=b_{i}\Big( C_{i+1,k}b_{k}(C_{i,k}C^{-1}_{i+1,k})b^{-1}_{k}\Big)b^{-1}_{i}= C_{i+1,k}b_{k}\cdot b_{i}C_{i,k}b^{-1}_{i}\cdot C^{-1}_{i+1,k}b^{-1}_{k},  
\end{align*}
from where \ref{it:add4} follows. For $j<i$, by using the same relations above and~\ref{it:add4} for $i=j$, we obtain
\begin{align*}
C_{j,k}=b_{i}\Big(C_{i+1,k}C^{-1}_{i,k}C_{j,k}b_{k}(C_{i,k}C^{-1}_{i+1,k})b^{-1}_{k}\Big)b^{-1}_{i}= (b^{-1}_{k}C^{-1}_{i,k}C_{i+1,k}b_{k})\cdot b_{i}C_{j,k}b^{-1}_{i}\cdot (C^{-1}_{i+1,k}C_{i,k}),  
\end{align*}
from where \ref{it:add4} follows. If $M=\mathbb K$, to prove~\ref{it:add4} for $i=j$, we use that $b_{i}$ commutes with $C_{i+1,k}$ and that $b_{k}b_{i}C_{i+1,k}C^{-1}_{i,k}=b_{i}b_{k}$ by relations~(\ref{it:puras8}) and~(\ref{it:puras6}), to obtain
\begin{align*}
C_{i,k}=&b_{i}\Big( C_{i+1,k}b_{k}(C_{i,k}C^{-1}_{i+1,k})^{-1}b^{-1}_{k}\Big)b^{-1}_{i}= C_{i+1,k}\Big(b_{k}b_{i}C_{i+1,k}C^{-1}_{i,k}\Big)(C_{i,k}C^{-1}_{i+1,k})^{-1}\Big(C_{i,k}C^{-1}_{i+1,k}b^{-1}_{i}b^{-1}_{k}\Big)\\
=&C_{i+1,k}b_{k}C_{i+1,k}\cdot b_{i}C^{-1}_{i,k}b^{-1}_{i}\cdot b^{-1}_{k},
\end{align*}
from where \ref{it:add4} follows. For $j<i$, by using the same relations above and~\ref{it:add4} for $i=j$, we obtain
\begin{align*}
C_{j,k}=b_{i}\Big(C_{i+1,k}C^{-1}_{i,k}C_{j,k}b_{k}(C_{i,k}C^{-1}_{i+1,k})^{-1}b^{-1}_{k}\Big)b^{-1}_{i}= (b^{-1}_{k}C^{-1}_{i+1,k}C_{i,k}b_{k})\cdot b_{i}C_{j,k}b^{-1}_{i}\cdot (C^{-1}_{i+1,k}C_{i,k}),
\end{align*}
from where \ref{it:add4} follows.
Now, item~\ref{it:add5} is relation~(\ref{it:puras6}) from Theorem~\ref{puras} if $M=\mathbb T$. If $M=\mathbb K$, by using relation~(\ref{it:puras6}) and relation~\ref{it:add4} we get

\begin{align*}
b_{j}=b_{i}\Big(b_{j}C_{i,j}C^{-1}_{i+1,j}\Big)b^{-1}_{i}=b_{i}b_{j}b^{-1}_{i}\cdot b_{i}\Big(C_{i,j}C^{-1}_{i+1,j}\Big)b^{-1}_{i}=b_{i}b_{j}b^{-1}_{i}\cdot b^{-1}_{j}\Big(C^{-1}_{i,j}C_{i+1,j}\Big)b_{j}    
\end{align*}
\end{proof}

Now, we will highlight some relations in $B_{n}(M)$ between the Artin generators and the generators of $P_{n}(M)$. 

\begin{remark}\label{Cij}  The elements $C_{j,k}$ can be seen in terms of the Artin generators. If $j=k$, the braid $C_{j,j}$ is defined to be the trivial braid. If $k=j+1$, we have $C_{j,j+1}=\sigma^{2}_{j}$. And if $j+1<k$, then\[C_{j,k}=\sigma_{k-1}\cdots\sigma_{j+1}\sigma^{2}_{j}\sigma_{j+1}\cdots\sigma_{k-1}.\]
\end{remark}

We also have some relations between the pure braids $a_{j}$ and $b_{j}$ and the Artin generators $\sigma_{i}$. These relations can be found in~\cite{GP} and can also be seen in Figure~\ref{fig:relacao}. There, we choose to present the relations in an arbitrary surface $M$ with generators $\rho_{i,r}$ to emphasize that the same relations are valid in any closed surface $M\neq \mathbb S^{2}$. We obtain the following
\begin{equation}\label{conj.rho}
\sigma^{-1}_{i}\rho_{j,r}\sigma_{i}=\left\{\begin{array}{ll}\sigma^{-2}_{i}\rho_{i+1,r},&\,\,\mbox{if}\,\,j=i\\ \rho_{i,r}\sigma^{2}_{i},&\,\,\mbox{if}\,\,j=i+1\\ \rho_{j,r},&\,\,\mbox{otherwise}\end{array}\right. 
\end{equation}
Moreover, it is possible to verify that the orientation of the edge $\epsilon_{r}$ does not modify the relations. In~(\ref{conj}) we explicit this relation if $M$ is the torus or the Klein bottle, using the notation of Theorem~\ref{puras}, \emph{i.e.} $a_{i}=\rho_{i,1}$ and $b_{i}=\rho^{-1}_{i,2}$.

\begin{figure}[h]
\begin{tikzpicture}[scale=0.7]
\centering
\draw[densely dotted] (3-8,0)--(2.42-8,1.76);
\draw[densely dotted] (2.42-8,1.76)--(0.91-8,2.84);
\draw[densely dotted] (0.91-8,2.84)--(-0.94-8,2.84);
\draw (-0.94-8,2.84)--(-2.43-8,1.74) node[sloped, pos=0.5, allow upside down]{\arrowIn};;
\draw[densely dotted] (-2.43-8,1.74)--(-3-8,-0.03);
\draw [densely dotted](-2.42-8,-1.79)--(-3-8,-0.03);
\draw[densely dotted] (-0.91-8,-2.87)--(-2.42-8,-1.79);
\draw[densely dotted] (-0.91-8,-2.87)--(0.94-8,-2.86);
\draw (2.43-8,-1.76)--(0.94-8,-2.86) node[sloped, pos=0.5, allow upside down]{\arrowIn};;
\draw[densely dotted] (2.43-8,-1.76)--(3-8,0);

\node at (-2-8,2.55) {\tiny$\epsilon_{r}$};
\node at (1.95-8,-2.65) {\tiny$\epsilon_{r}$};
\node at (-1.55-8,-0.3) {\tiny $i$};
\node at (0.35-8,-0.3) {\tiny $j$};

\draw[blue] (0.5-8,0).. controls (0.75-8,1.25) and (-1.25-8,0.75)..  (-1.34-8,2.53) node[sloped, pos=0.9, allow upside down]{\arrowIn};;

\draw[blue] (2.03-8,-2.04).. controls  (2.25-8,-0.5) and (0.2-8,-0.75)  .. (0.5-8,0) node[sloped, pos=0.5, allow upside down]{\arrowIn};;

\draw (-0.75-8,0).. controls  (-0.75-8,0.25) and (-1.5-8,0.25)  .. (-1.5-8,0) node[sloped, pos=0.5, allow upside down]{\arrowIn};;
\draw (-1.5-8,0).. controls  (-1.5-8,-0.25) and (-0.75-8,-0.25)  .. (-0.75-8,0) node[sloped, pos=0.5, allow upside down]{\arrowIn};;



{\filldraw[blue] (2.03-8,-2.04) circle (1.5pt);}
{\filldraw[blue] (-1.34-8,2.53) circle (1.5pt);}

{\filldraw[black] (0-8,0) circle (1pt);}
{\filldraw[black] (0.5-8,0) circle (1pt);}
{\filldraw[black] (1.25-8,0) circle (1pt);}
{\filldraw[black] (-0.75-8,0) circle (1pt);}
{\filldraw[black] (-1.5-8,0) circle (1pt);}
{\filldraw[black] (2.25-8,0) circle (1pt);}
{\filldraw[black] (-8-2.25,0) circle (1pt);}


\draw[densely dotted] (3,0)--(2.42,1.76);
\draw[densely dotted] (2.42,1.76)--(0.91,2.84);
\draw[densely dotted] (0.91,2.84)--(-0.94,2.84);
\draw[densely dotted] (-2.43,1.74)--(-3,-0.03);
\draw[densely dotted] (-2.42,-1.79)--(-3,-0.03);
\draw[densely dotted] (-0.91,-2.87)--(-2.42,-1.79);
\draw[densely dotted] (-0.91,-2.87)--(0.94,-2.86);
\draw (2.43,-1.76)--(0.94,-2.86) node[sloped, pos=0.5, allow upside down]{\arrowIn};;
\draw[densely dotted] (2.43,-1.76)--(3,0);

\node at (-2,2.55) {\tiny$\epsilon_{r}$};
\node at (1.95,-2.65) {\tiny$\epsilon_{r}$};
\node at (-0.3,-0.2) {\tiny $i$};

\draw[white,line width=2pt] (0,1.5).. controls (-0.5,1.5) and (-1,0.75)  .. (0,0) node[sloped, pos=0.9, allow upside down]{\arrowIn};;
\draw (0,1.5).. controls (-0.5,1.5) and (-1,0.75)  .. (0,0) node[sloped, pos=0.8, allow upside down]{\arrowIn};;
\draw[white,line width=4pt] (0.5,0).. controls (0.75,1.25) and (-1.25,0.75)..  (-1.34,2.53) node[sloped, pos=0.9, allow upside down]{\arrowIn};;
\draw[blue] (0.5,0).. controls (0.75,1.25) and (-1.25,0.75)..  (-1.34,2.53) node[sloped, pos=0.9, allow upside down]{\arrowIn};;
\draw[blue] (2.03,-2.04).. controls  (2.25,-0.5) and (-0.25,-1).. (0,-0.5) node[sloped, pos=0.75, allow upside down]{\arrowIn};;
\draw[blue] (0,-0.5).. controls  (0.5,-0.1) and (0.5,-0.1)  .. (0.5,0);

\draw[white,line width=3pt] (0,0).. controls (1,0.75) and (0.5,1.5).. (0,1.5);
\draw (0,0).. controls (1,0.75) and (0.5,1.5).. (0,1.5);

\draw (-0.94,2.84)--(-2.43,1.74) node[sloped, pos=0.5, allow upside down]{\arrowIn};;

{\filldraw[blue] (2.03,-2.04) circle (1.5pt);}
{\filldraw[blue] (-1.34,2.53) circle (1.5pt);}

{\filldraw[black] (0,0) circle (1pt);}
{\filldraw[black] (0.5,0) circle (1pt);}
{\filldraw[black] (1.25,0) circle (1pt);}
{\filldraw[black] (-0.75,0) circle (1pt);}
{\filldraw[black] (-1.5,0) circle (1pt);}
{\filldraw[black] (2.25,0) circle (1pt);}
{\filldraw[black] (-2.25,0) circle (1pt);}

\draw[densely dotted] (11,0)--(10.42,1.76);
\draw[densely dotted] (10.42,1.76)--(8.91,2.84);
\draw[densely dotted] (8.91,2.84)--(8-0.94,2.84);
\draw (8-0.94,2.84)--(8-2.43,1.74) node[sloped, pos=0.5, allow upside down]{\arrowIn};;
\draw[densely dotted] (8-2.43,1.74)--(8-3,-0.03);
\draw[densely dotted] (8-2.42,-1.79)--(8-3,-0.03);
\draw[densely dotted] (8-0.91,-2.87)--(8-2.42,-1.79);
\draw[densely dotted] (8-0.91,-2.87)--(8.94,-2.86);
\draw (10.43,-1.76)--(8.94,-2.86) node[sloped, pos=0.5, allow upside down]{\arrowIn};;
\draw[densely dotted] (10.43,-1.76)--(11,0);

\node at (8-2,2.55) {\tiny$\epsilon_{r}$};
\node at (9.95,-2.65) {\tiny$\epsilon_{r}$};
\node at (8,0.3) {\tiny$i$};

\node at (4,0) {$\cong$};
\node at (-8,-3.75) {(a) $\rho_{j,r}\sigma_{i}=\sigma_{i}\rho_{j,r}$};

\node at (4,-3.75) {(b) $\sigma^{-1}_{i}\rho_{i,r}\sigma_{i}=\sigma^{-2}_{i}\rho_{i+1,r}$};


\draw[blue] (8.5,0).. controls (7.25,-1.25) and (8,1.5)..  (8-1.34,2.53) node[sloped, pos=0.8, allow upside down]{\arrowIn};;

\draw[blue] (8+2.03,-2.04).. controls  (8+2.25,-0.5) and (8+0.2,-0.75)  .. (8+0.5,0) node[sloped, pos=0.5, allow upside down]{\arrowIn};;


{\filldraw[blue] (8+2.03,-2.04) circle (1.5pt);}
{\filldraw[blue] (8-1.34,2.53) circle (1.5pt);}

{\filldraw[black] (8,0) circle (1pt);}
{\filldraw[black] (8.5,0) circle (1pt);}
{\filldraw[black] (9.25,0) circle (1pt);}
{\filldraw[black] (8-0.75,0) circle (1pt);}
{\filldraw[black] (8-1.5,0) circle (1pt);}
{\filldraw[black] (10.25,0) circle (1pt);}
{\filldraw[black] (8-2.25,0) circle (1pt);}

\end{tikzpicture}
\caption{Relation}\label{fig:relacao}
\end{figure}
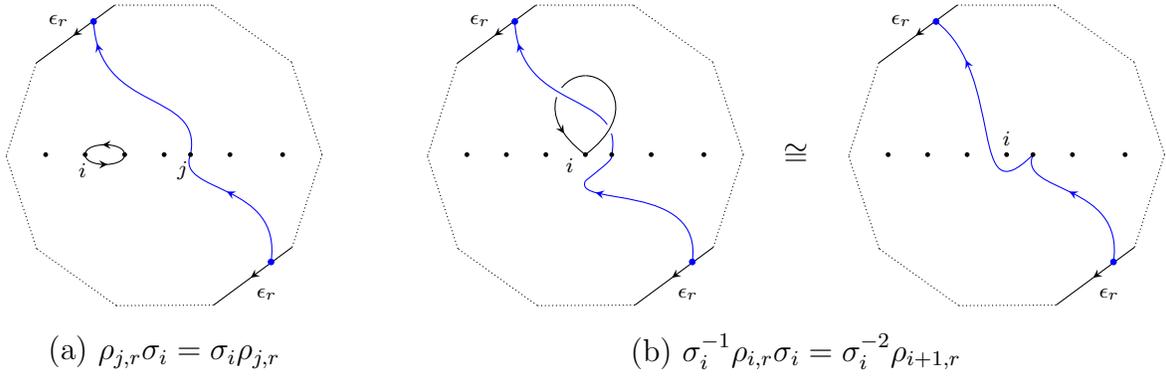

\begin{equation}\label{conj}
\begin{array}{lll}\sigma^{-1}_{i}a_{j}\sigma_{i}=\left\{\begin{array}{ll}\sigma^{-2}_{i}a_{i+1},&\,\,\mbox{if}\,\,j=i\\a_{i}\sigma^{2}_{i},&\,\,\mbox{if}\,\,j=i+1\\a_{j},&\,\,\mbox{otherwise}\end{array}\right. & \quad\quad &\sigma^{-1}_{i}b_{j}\sigma_{i}=\left\{\begin{array}{ll}b_{i+1}\sigma^{2}_{i},&\,\,\mbox{if}\,\,j=i\\\sigma^{-2}_{i}b_{i},&\,\,\mbox{if}\,\,j=i+1\\b_{j},&\,\,\mbox{otherwise}\end{array}\right.
\end{array}
\end{equation}

In the following Lemma, we present some more important relations in $B_{n}(\mathbb T)$ and $B_{n}(\mathbb K)$  that will play a key role in the proof of Theorem~\ref{BNS-Klein}. In the case of the torus, some of these relations are already in Theorem~\ref{puras}, but we rewrite them for completeness. Denote by

\begin{equation}
\alpha_{j,i}=\prod^{j}_{k=i}a_{j+i-k}\quad \mbox{and}\quad \beta_{j,i}=\prod^{j}_{k=i}b_{j+i-k},    
\end{equation}
for $1\leq i \leq j \leq n$. If $i>j$, it will be convenient to deﬁne $\alpha_{j,i}$ and $\beta_{j,i}$ to be the trivial braid.

\begin{lema}\label{rel} Let $M$ be the torus $\mathbb T$ or the Klein bottle $\mathbb K$. Using the presentation given in Theorem~\ref{puras}, 
the following relations are valid in $B_{n}(M)$.

\begin{enumerate}[label=(\subscript{R}{{\arabic*}})]
\item\label{it:a1} $a_{i+1}a_{i}\cdot \sigma_{i}=\sigma_{i}\cdot a_{i+1}a_{i},\,\,(1\leq i < n)$;

\item\label{it:a2} $\displaystyle \alpha_{j,i}\cdot \sigma_{k}=\sigma_{k}\cdot \alpha_{j,i},\,\,(1\leq i\leq k< j \leq n)$;

\item\label{it:a3} $\displaystyle \alpha_{j,i}\cdot C_{k,t}=C_{k,t}\cdot \alpha_{j,i},\,\,(1\leq i\leq k< t\leq j \leq n)$;

\item\label{it:b1} $\left\{ \begin{array}{llr} b_{i+1}b_{i}\cdot \sigma_{i}=\sigma_{i}\cdot b_{i+1}b_{i},&(1\leq i < n),&\,\,\mbox{if}\,\,,M=\mathbb T,\\ b_{i+1}b_{i}\cdot\sigma_{i}=\sigma^{-1}_{i}\cdot b_{i+1}b_{i}, &(1\leq i < n),&\,\,\mbox{if}\,\,M=\mathbb K;\end{array}\right.$

\item\label{it:b2} $ \left\{ \begin{array}{llr}\displaystyle\beta_{j,i}\cdot \sigma_{k}=\sigma_{k}\cdot \beta_{j,i},&(1\leq i\leq k< j \leq n),&\,\,\mbox{if}\,\,M=\mathbb T,\\
\displaystyle\beta_{j,i}\cdot \sigma_{k}=\sigma^{-1}_{k}\cdot \beta_{j,i},&(1\leq i\leq k<j \leq n),&\,\,\mbox{if}\,\,M=\mathbb K;\end{array}\right.$

\item\label{it:b3} $\left\{ \begin{array}{llr}\displaystyle\beta_{j,i}\cdot C_{k,t}=C_{k,t}\cdot \beta_{j,i},&(1\leq i\leq k<t\leq j \leq n)&\,\,\mbox{if}\,\,,M=\mathbb T,\\
\displaystyle\beta_{j,i}\cdot C_{k,t}=C^{-1}_{k,t}\cdot \beta_{j,i},&(1\leq i\leq k<t\leq j \leq n)&\,\,\mbox{if}\,\,,M=\mathbb K;\end{array}\right.$

\item\label{it:b4} $\left\{\begin{array}{llr}\beta_{j,i}\cdot b_{j}=b_{j}\cdot\beta_{j,i},&(1\leq i< j \leq n ),&\,\,if\,\,M=\mathbb T,\\
\beta_{j,i}\cdot b_{j}C_{i,j}=b_{j}\cdot \beta_{j,i},&(1\leq i< j \leq n ),&\,\,\mbox{if}\,\,M=\mathbb K.\end{array}\right.$
\end{enumerate}
\end{lema}

\begin{proof} Relation~\ref{it:a1} follows from~(\ref{conj}) and relation~(\ref{it:puras1}) from Theorem~\ref{puras}:
\begin{align*}
    \sigma^{-1}_{i}(a_{i+1}a_{i})\sigma_{i}=(a_{i}\sigma^{2}_{i})(\sigma^{-2}_{i}a_{i+1})=a_{i}a_{i+1}=a_{i+1}a_{i},
\end{align*}
and relation~\ref{it:a2} (resp. \ref{it:a3}) is an immediate consequence of~\ref{it:a1} and~(\ref{conj}) (resp.~\ref{it:a2} and Remark~\ref{Cij}). Similarly, relation~\ref{it:b1} follows from~(\ref{conj}), relation~(\ref{it:puras6}) from Theorem~\ref{puras} and the fact that $\sigma^{2}_{i}=C_{i,i+1}$ by Remark~\ref{Cij}:
\[\left\{\begin{array}{lr}
\sigma^{-1}_{i}(b_{i+1}b_{i})\sigma_{i}=\sigma^{-1}_{i}(b_{i}b_{i+1})\sigma_{i}=
(b_{i+1}\sigma^{2}_{i})(\sigma^{-2}_{i}b_{i})=b_{i+1}b_{i},&\,\,\mbox{if}\,\,M=\mathbb T\\
\sigma^{-1}_{i}(b_{i+1}b_{i})\sigma_{i}=\sigma^{-1}_{i}(b_{i}b_{i+1}\sigma^{2}_{i})\sigma_{i}=
(b_{i+1}\sigma^{2}_{i})(\sigma^{-2}_{i}b_{i})\sigma^{2}_{i}=b_{i+1}b_{i}\sigma^{2}_{i},&\,\,\mbox{if}\,\,M=\mathbb K
%
\end{array}\right.\]
and relation~\ref{it:b2} (resp. \ref{it:b3}) is an immediate consequence of~\ref{it:b1} and~(\ref{conj}) (resp.~\ref{it:b2} and Remark~\ref{Cij}).

If $M=\mathbb T$, relation~\ref{it:b4} follows directly from~(\ref{it:puras6}) of Theorem~\ref{puras}, however if $M=\mathbb K$ we need a few more steps. Fix $1<j \leq n$ and let us show \ref{it:b4} by (descending) induction on $1 \leq i<j$. For $i=j-1$, \ref{it:b4} follows directly from relation~(\ref{it:puras6}) of Theorem~\ref{puras}. Suppose \ref{it:b4} is true for $1 < i<j \leq n$ and let us show it to $i-1$. By using the induction hypothesis and relation~(\ref{it:puras8}) of Theorem~\ref{puras} we obtain
\[
\beta_{j,i-1}^{-1}b_j\beta_{j,i-1}=b_{i-1}^{-1}\beta_{j,i}^{-1}b_j\beta_{j,i}b_{i-1}=b_{i-1}^{-1}b_jC_{i,j}b_{i-1}=(b_jC_{i-1,j}C_{i,j}^{-1})C_{i,j}=b_jC_{i-1,j},
\] so $b_j\beta_{j,i-1}=\beta_{j,i-1}b_jC_{i-1,j}$, as desired. This completes the induction step and, therefore, the proof of the Theorem.
\end{proof}

\begin{lema}\label{caminho} Let $M$ be the torus $\mathbb T$ or the Klein bottle $\mathbb K$, and $n\geq 3$. The following relations are valid in $P_{n}(M)$.
\begin{enumerate}[label=(\subscript{P}{{\arabic*}})]

\item\label{it:c1} $\left\{\begin{array}{lr}b^{-1}_{n}a_{i}b_{n}=\displaystyle\left\{\begin{array}{lr} C_{i,n}C^{-1}_{i+1,n}a_{i},&(1\leq i<n)\\
a_{n}C^{-1}_{1,n},&(i=n)\end{array}\right. &\,\,\mbox{if}\,\,M=\mathbb T\\

b^{-1}_{n}a_{i}b_{n}=\displaystyle\left\{\begin{array}{lr} C_{i,n}C^{-1}_{i+1,n}a_{i},&(1\leq i<n)\\
C_{1,n}a^{-1}_{n},&(i=n) \end{array}\right. &\,\,\mbox{if}\,\,M=\mathbb K\end{array}\right.$
  
    \item\label{it:c2} $\left\{\begin{array}{lr} b^{-1}_{n}C_{i,n}b_{n}=\displaystyle\left\{\begin{array}{lr} \beta_{n-1,i} C_{i,n} \beta_{n-1,i}^{-1},&(1<i<n)\\
 \beta_{n-1,3}b_{1}\beta_{n-1,2}C^{-1}_{2,n}C_{3,n}\beta^{-1}_{n-1,2}\delta b_{n}b^{-1}_{1}\beta^{-1}_{n,3},&(i=1)\end{array}\right.&\,\,\mbox{if}\,\,M=\mathbb T\\
b^{-1}_{n}C_{i,n}b_{n}=\displaystyle\left\{\begin{array}{lr} \beta_{n-1,i} C^{-1}_{i,n} \beta_{n-1,i}^{-1},&(1<i<n)\\
\beta_{n-1,3}b_{1}\delta^{-1}\beta_{n-1,2}C_{3,n}C^{-1}_{2,n}\beta_{n-1,2}^{-1}b_{n}\delta b^{-1}_{1}\beta_{n,3}^{-1},&(i=1)\end{array}\right.&\,\,\mbox{if}\,\,M=\mathbb K
\end{array}\right.$

\item\label{it:c3} if $M=\mathbb T$, then $a^{-1}_{n}b_{i}a_{n}=\displaystyle\left\{\begin{array}{lr} C^{-1}_{i,n}C_{i+1,n}b_{i},&(1\leq i<n)\\
b_{n}C_{1,n},&(i=n)\end{array}\right.$
  
\item\label{it:c4} if $M=\mathbb T$, then  $a^{-1}_{n}C_{i,n}a_{n}=\displaystyle\left\{\begin{array}{lr} \alpha_{n-1,i} C_{i,n} \alpha_{n-1,i}^{-1},&(1<i<n)\\
\alpha_{n-1,3}a_{1}\bar{\delta}\alpha_{n-1,2}C_{2,n}C^{-1}_{3,n}\alpha^{-1}_{n-1,2}a_{n}a^{-1}_{1}\alpha^{-1}_{n,3},&(i=1)\end{array}\right.$

\end{enumerate}

where $\delta=C_{1,n}C^{-1}_{2,n}C_{3,n}$ and $\bar{\delta}=C_{3,n}C^{-1}_{2,n}C_{1,n}$.
\end{lema}

\begin{proof}
If $1\leq i <n$, relation~\ref{it:c1} is a consequence of relations~(\ref{it:puras2}) and~(\ref{it:puras3}) from Theorem~\ref{puras}:
\begin{align*}
a^{-1}_{i}b_{n}a_{i}=b_{n}a_{n}C^{-1}_{i,n}C_{i+1,n}a^{-1}_{n}=b_{n}a^{-1}_{i}C_{i+1,n}C^{-1}_{i,n}a_{i}    
\end{align*}
therefore, $a_{i}b_{n}=b_{n}C_{i,n}C^{-1}_{i+1,n}a_{i}$. The case $i=n$ comes from relation~(\ref{it:puras5}) of Theorem~\ref{puras}.

To prove relation~\ref{it:c2}, if $1<i<n$, by Lemma~\ref{rel}~\ref{it:b3} we have 
\begin{align}\label{anterior}
\displaystyle C_{i,n}b_{n}=C_{i,n}b_{n}\beta_{n-1,i}\beta^{-1}_{n-1,i}=C_{i,n}\beta_{n,i}\beta^{-1}_{n-1,i}=\left\{\begin{array}{lr}b_{n}\beta_{n-1,i}C_{i,n}\beta^{-1}_{n-1,i}, & M=\mathbb T\\
b_{n}\beta_{n-1,i}C^{-1}_{i,n}\beta^{-1}_{n-1,i}, & M=\mathbb K.\end{array}\right.
\end{align}
For the case $i=1$, we need some additional steps. In the following, we use Lemma~\ref{rel}, Theorem~\ref{puras}~(\ref{it:puras6}), Proposição~\ref{prop:add}~\ref{it:add4} and~\ref{it:add5}, and also~(\ref{anterior}). If $M=\mathbb T$, we have
\begin{align*}
C_{1,n}b_{n}\beta_{n,3}b_{1}=&C_{1,n}b_{n}\beta_{n,1}(b^{-1}_{1}b^{-1}_{2}b_{1})=\beta_{n,1}C_{1,n}b_{n}b^{-1}_{2}=\beta_{n,3}b_{1}b_{2}C_{1,n}b_{n}b^{-1}_{2}=\beta_{n,3}b_{1}(b^{-1}_{n}C^{-1}_{3,n}C_{2,n}b_{n}\delta)b_{n}\\
=&\beta_{n,3}b_{1}(\beta_{n-1,3}C^{-1}_{3,n}\beta^{-1}_{n-1,3}\beta_{n-1,2}C_{2,n}\beta^{-1}_{n-1,2})\delta b_{n}=\beta_{n,3}b_{1}\beta_{n-1,3}C^{-1}_{3,n}b_{2}C_{2,n}\beta^{-1}_{n-1,2}\delta b_{n}\\
=&b_{n}\cdot \beta_{n-1,3}b_{1}\beta_{n-1,2}C^{-1}_{3,n}C_{2,n}\beta^{-1}_{n-1,2}\delta b_{n}.
\end{align*}

Similarly, if $M=\mathbb K$, we obtain
\begin{align*}
C_{1,n}b_{n}\beta_{n,3}b_{1}=C_{1,n}b_{n}\beta_{n,1}(b^{-1}_{1}b^{-1}_{2}b_{1})=\beta_{n,1}C^{-1}_{1,n}b_{n}C_{1,n}(b^{-1}_{1}b^{-1}_{2}b_{1})=\beta_{n,3}(b_{2}b_{1})C^{-1}_{1,n}b_{n}C_{1,n}(b^{-1}_{1}b^{-1}_{2}b_{1})\\
=\beta_{n,3}(b_{1}b_{2}C_{1,2})C^{-1}_{1,n}b_{n}C_{1,n}(C^{-1}_{1,2}b^{-1}_{2})=\beta_{n,3}b_{1}b_{2}C^{-1}_{1,n}b_{n}C_{1,n}b^{-1}_{2}
=\beta_{n,3}b_{1}(\delta^{-1}\cdot b^{-1}_{n}C^{-1}_{3,n}C_{2,n}b_{n}\cdot b_{n}\delta)\\
=\beta_{n,3}b_{1}\delta^{-1}( \beta_{n-1,3}C_{3,n}\beta^{-1}_{n-1,3} \beta_{n-1,2}C^{-1}_{2,n}\beta^{-1}_{n-1,2}) b_{n}\delta
=\beta_{n,3}b_{1}\delta^{-1} \beta_{n-1,3}C_{3,n}b_{2}C^{-1}_{2,n}\beta^{-1}_{n-1,2} b_{n}\delta\\
=b_{n}\cdot\beta_{n-1,3}b_{1}\delta^{-1} \beta_{n-1,2}C_{3,n}C^{-1}_{2,n}\beta^{-1}_{n-1,2} b_{n}\delta.
\end{align*}

In a similar way as relation~\ref{it:c1}, to prove~\ref{it:c3} we use relations~(\ref{it:puras7}) and~(\ref{it:puras8}) of Theorem~\ref{puras}, if $1\leq i < n$ and relation~(\ref{it:puras5}) if $i=n$. To prove~\ref{it:c4}, the first part is analogous to the proof of~\ref{it:c2}.
Now, if $i=1$, using Lemma~\ref{rel}, Theorem~\ref{puras}~(\ref{it:puras1}), the previous case and also Lema~\ref{prop:add}~\ref{it:add3}, we obtain
\begin{align*}
C_{i,n}a_{n}\alpha_{n,3}a_{1}=&C_{1,n}a_{n}\alpha_{n,1}a^{-1}_{2}=\alpha_{n,1}C_{1,n}a_{n}a^{-1}_{2}=\alpha_{n,3}a_{1}a_{2}C_{1,n}a_{n}a^{-1}_{2}=\alpha_{n,3}a_{1}(\bar{\delta}a^{-1}_{n}C_{2,n}C^{-1}_{3,n}a_{n})a_{n}\\
=&\alpha_{n,3}a_{1}\bar{\delta}(\alpha_{n-1,2}C_{2,n}\alpha^{-1}_{n-1,2}\alpha_{n-1,3}C^{-1}_{3,n}\alpha^{-1}_{n-1,3})a_{n}=\alpha_{n,3}a_{1}\bar{\delta}(\alpha_{n-1,2}C_{2,n}a^{-1}_{2}C^{-1}_{3,n}\alpha^{-1}_{n-1,3})a_{n}\\
=&a_{n}\cdot\alpha_{n-1,3}a_{1}\bar{\delta}\alpha_{n-1,2}C_{2,n}C^{-1}_{3,n}\alpha^{-1}_{n-1,2}a_{n},
\end{align*}
which concludes the proof.
\end{proof}

To obtain the character sphere of a group, we need its abelianization, which we present in the following proposition. The proof is straightforward from the presentation of those groups, and a precise proof can be found in~\cite{GGOP} and its references.

\begin{prop}\label{prop:ab}~\cite[Corollary 7 and Lemma 15]{GGOP} Let $M$ be a compact surface without boundary, and let $n\in\mathbb N$. The group $P_{n}(M)/\gamma_{2}(P_{n}(M))$ is isomorphic to:
\begin{enumerate}
    \item\label{it:ab1} $\mathbb Z_{2} \oplus\mathbb Z^{n(n-3)}$ if $M =\mathbb  S^{2}$ and $n\geq 3$.

    \item\label{it:ab2} $\mathbb Z^{2gn}$ if $M$ is an orientable surface of genus $g\geq 1$.
    
\item\label{it:ab3} $\mathbb Z^{n}_{2}\oplus\mathbb Z^{(g-1)n}$ if $M$ is a non-orientable surface of genus $g\geq 1$.
\end{enumerate}
\end{prop}

\begin{remark}\label{rem:basis} If $M$ is the torus or the Klein bottle, it is easy to see that for all $1\leq i<j\leq n$, $C_{i,j}\in\gamma_{2}(P_{n}(M))$  by Theorem~\ref{puras}~(\ref{it:puras2}). In fact, $C_{i,j}\in\gamma_{2}(P_{n}(M))$ for all closed surfaces $M\neq \mathbb S^{2}$, and the generators $\rho_{i,r}$ form a basis for $P_{n}(M)/\gamma_{2}(P_{n}(M))$~\cite{GGOP}. Now, adjusting to the notation of Theorem~\ref{puras}, if $M=\mathbb T$ then $\left\{a_{i},b_{i}\,:\,1\leq i \leq n\right\}$ is a basis for the free abelian group $P_{n}(\mathbb T)/\gamma_{2}(P_{n}(\mathbb T))$. On the other hand, if $M=\mathbb K$, then by Theorem~\ref{puras}~(\ref{it:puras5}), we have that $\left\{a_{i}\,:\,1\leq i \leq n\right\}$ is a basis for the torsion part of $P_{n}(\mathbb K)/\gamma_{2}(P_{n}(\mathbb K))$ and $\left\{b_{i}\,:\,1\leq i \leq n\right\}$ is a basis for the free part of $P_{n}(\mathbb K)/\gamma_{2}(P_{n}(\mathbb K))$. 
\end{remark}

It is also straightforward to obtain the abelianization of $B_{n}(M)$, by looking to the presentation of each group. This can be found in~\cite{BGG, GG3,GG5, GP}, and we state the result below.

\begin{prop}\label{prop:Abel} Let $M$ be a compact surface without boundary, and let $n\in\mathbb N$. The group $B_{n}(M)/\gamma_{2}(B_{n}(M))$ is isomorphic to:
\begin{enumerate}
    \item\label{it:Abel1} $\mathbb Z_{2(n-1)}$ if $M =\mathbb  S^{2}$ and $n\geq 2$.

    \item\label{it:Abel2} $\mathbb Z_{2}\oplus\mathbb Z^{2g}$ if $M$ is an orientable surface of genus $g\geq 1$.

 \item\label{it:Abel3} $\mathbb Z^{2}_{2}\oplus\mathbb Z^{g-1}$ if $M$ is a non-orientable surface of genus $g\geq 1$.
\end{enumerate}
\end{prop}

\begin{remark}\label{rem:sigma} It follows from the Artin relations that the $\gamma_{2}(B_{n}(M))$-cosets of $\sigma_{1},\ldots, \sigma_{n-1}$ in $B_{n} (M)/\gamma_{2}(B_{n} (M))$ are all identiﬁed to a single element, which we denote by $\sigma$. In particular, if $M\neq\mathbb S^{2}$, the element $\sigma$ is a torsion element of order $2$ of $B_{n} (M)/\gamma_{2}(B_{n} (M))$, by Remarks~\ref{Cij} and~\ref{rem:basis}.
\end{remark}

\begin{remark}\label{rem:a,b} Let $M$ be the torus $\mathbb T$ or the Klein bottle $\mathbb K$. If $2\leq i \leq n$, by~(\ref{conj}) we have 
\begin{equation}
 a_{i}=(\sigma_{i-1}\cdots\sigma_{1})\cdot a_{1}\cdot (\sigma_{1}\cdots\sigma_{i-1})\quad\mbox{and}\quad b_{i}=(\sigma^{-1}_{i-1}\cdots\sigma^{-1}_{1})\cdot b_{1}\cdot (\sigma^{-1}_{1}\cdots\sigma^{-1}_{i-1}),
\end{equation}
therefore, by Remark~\ref{rem:sigma}, it follows that the $\gamma_{2}(B_{n}(M))$-cosets of $a_{1},\ldots, a_{n}$ (resp. $b_{1},\ldots, b_{n}$) in $B_{n} (M)/\gamma_{2}(B_{n} (M))$ are all identiﬁed to a single element, which we denote by $a$ (resp. $b$). If $M=\mathbb T$, then $a$ and $b$ are torsion free, and if $M=\mathbb K$, then $a$ has order $2$ and $b$ is torsion free.\end{remark}

\subsection{Pure braid groups of $\mathbb T$ and $\mathbb K$ as iterated semi-direct products}

In this section, we give another presentation for some pure braid groups of the torus and the Klein bottle. If $M$ is a surface without boundary, due the work of E.~Fadell and L.~Neuwirth~\cite{FN}, we have the following short exact sequence of pure braid groups:
\begin{equation}\label{seq:FN}
	\xymatrix{
		1 \ar[r] & \pi_{1}(M\setminus\left\{n-1\,\mbox{pts}\right\}) \ar[r]^-{\iota} & P_{n}(M) \ar[r]^-{{p_{j}}_{\ast}} & P_{n-1}(M) \ar[r] & 1,
	}		
\end{equation} 
where $n> 3$ if $M$ is the sphere $\mathbb S^{2}$~\cite{Fa, FVB}, $n> 2$ if $M$ is the projective plane $\mathbb RP^{2}$~\cite{FVB}, and $n\geq 2$ otherwise~\cite{FN}. The homomorphism $\iota$ is the inclusion and ${p_{j}}_{\ast}$ can be interpreted as the homomorphism that erase the $j$-th string. One important result of Fadell and Neuwirth~\cite{FN} guarantees that for $M$ either the torus or the Klein bottle, due the existence of a non-vanishing vector field in $M$, the short exact sequence~(\ref{seq:FN}) splits for all $n$; therefore, $P_{n}(M)$ may be decomposed as an iterated semi-direct product, which is not true in general (see~\cite[Theorem 2]{GG4}).
Using the explicit algebraic description of the section of ${p_{n}}_{\ast}$ (see \cite[Proposition 5.1]{GP} and~\cite[Proposition 2.2.1]{P}), we can obtain the desired iterated semi-direct product for all $n$. In the following, we explicitly describe these structures for $P_2(\mathbb T)$, $P_3(\mathbb T)$, $P_4(\mathbb T)$ and $P_2(\mathbb K)$. We will make use of these alternative presentations in the proof of Theorem \ref{BNS-Klein}, through Lemmas~\ref{BNS.P2},~\ref{lem:P2},~\ref{BNS.P3} and~\ref{BNS.P4}.
\begin{prop}\label{prop:Pn} The following assertions hold:
\begin{enumerate}[label=(\roman*)]
\item\label{it:P2} The pure braid group $P_{2}(\mathbb T)$ is isomorphic to $G_{2}(\mathbb T)=F_{2}\times \mathbb Z\times\mathbb Z$, the direct product of the free group $F_{2}=\langle x,y\rangle$ with $\mathbb Z\times \mathbb Z=\langle a,b\rangle$.

\item\label{it:P3}  The pure braid group $P_{3}(\mathbb T)$ is isomorphic to $G_{3}(\mathbb T)=F_{3}\rtimes G_{2}(\mathbb T)$, the semi-direct product of the free group $F_{3}=\langle u,v,w\rangle$ with the group $G_{2}(\mathbb T)$, defined above, equipped with the following action
\begin{enumerate}[label=(\alph*)]
\begin{multicols}{2}
\item $a^{-1}za=z$, if $z=u,v,w$;

\item $x^{-1}zx=\left\{\begin{array}{ll}u& \mbox{if}\,\,z=u,\\
 u^{-1}vuw^{-1}& \mbox{if}\,\,z=v,\\
w& \mbox{if}\,\,z=w;\end{array}\right.$

\item $b^{-1}zb=z$, if $z=u,v,w$;

\item $y^{-1}zy=\left\{\begin{array}{ll} v^{-1}uvw& \mbox{if}\,\,z=u,\\
v& \mbox{if}\,\,z=v,\\
w& \mbox{if}\,\,z= w.\end{array}\right.$.
\end{multicols}
\end{enumerate}

\item\label{it:P4} The pure braid group $P_{4}(\mathbb T)$ is isomorphic to $G_{4}(\mathbb T)=F_{4}\rtimes G_{3}(\mathbb T)$, the semi-direct product of the free group $F_{4}=\langle \bar{u},\bar{v},w_{2},w_{3}\rangle$ with the group $G_{3}(\mathbb T)$, defined above, equipped with the following action
\begin{enumerate}[label=(\alph*)]
\begin{multicols}{2}
\item $a^{-1}za=z$, if $z=\bar{u},\bar{v},w_{2},w_{3}$;\\

\item $x^{-1}zx=\left\{\begin{array}{l}\bar{u}\quad(z=\bar{u}),\\
 \bar{u}^{-1}\bar{v}\bar{u}w^{-1}_{2}\quad(z=\bar{v}),\\
w_{i}\quad (z=w_{i},\, i=2,3);\end{array}\right.$\\

\item $b^{-1}zb=z$, if $z=\bar{u},\bar{v},w_{2},w_{3}$;\\

\item $y^{-1}zy=\left\{\begin{array}{l} \bar{v}^{-1}\bar{u}\bar{v}w_{2}\quad(z=\bar{u}),\\
\bar{v}\quad(z=\bar{v}),\\
w_{i}\quad(z= w_{i},\,i=2,3);\end{array}\right.$

\item $u^{-1}zu=\left\{\begin{array}{ll}\bar{u}& (z=\bar{u}),\\
 \bar{v}\bar{u}w^{-1}_{2}\bar{u}^{-1}& (z=\bar{v}),\\
 w_{3}\bar{u}^{-1}w_{2}w^{-1}_{3}\bar{u}& (z=w_{2}),\\
w_{3}& (z=w_{3});\end{array}\right.$

\item\label{it:f} $v^{-1}zv=\left\{\begin{array}{ll} \bar{u}\bar{v}w_{2}\bar{v}^{-1}& (z=\bar{u}),\\
\bar{v}& (z=\bar{v}),\\
\bar{v}^{-1}w^{-1}_{3}w_{2}\bar{v}w_{3}& (z= w_{2}),\\
w_{3}& (z= w_{3});\end{array}\right.$

\item $w^{-1}zw=\left\{\begin{array}{ll} w_{3}w^{-1}_{2}zw_{2}w^{-1}_{3}& (z=\bar{u},\bar{v}),\\
w_{3}w_{2}w^{-1}_{3}& (z= w_{2}),\\
w_{3}& (z= w_{3}).\end{array}\right.$
\end{multicols}
\end{enumerate}
\end{enumerate}
\end{prop}

\begin{proof} We use the presentation of $P_{n}(\mathbb T)$ given in Theorem~\ref{puras}. For~\ref{it:P2}, consider the homomorphism  $\phi:G_{2}(\mathbb T)\longrightarrow P_{2}(\mathbb T)$ defined in the generators as:
\begin{align}\label{phi}\phi:
\left\{\begin{array}{ccl}
x&\longmapsto& a_{2}\\
y&\longmapsto& b_{2}\\
a&\longmapsto& a_{1}a_{2}\\
b&\longmapsto& b_{2}b_{1}
   \end{array}\right.
\end{align}
It is straightforward to prove that $\phi$ is well defined and that $\phi$ is bijetive, by using that $Z(P_{2}(\mathbb T))=\langle (a_{1}a_{2}), (b_{1}b_{2})\rangle$~\cite{B1,PR}. Similarly, for items~\ref{it:P3} and~\ref{it:P4}, consider the homomorphisms $\Phi:G_{3}(\mathbb T)\longrightarrow P_{3}(\mathbb T)$ and $\Psi:G_{4}(\mathbb T)\longrightarrow P_{4}(\mathbb T)$ defined as:
\begin{align}\label{Phi}\Phi:
\left\{\begin{array}{ccl}
u&\longmapsto& a_{3}\\
v&\longmapsto& b_{3}\\
w&\longmapsto& C_{2,3}\\
   \end{array}\right.&\quad&
\Phi:
\left\{\begin{array}{ccl}
x&\longmapsto& a_{2}a_{3}\\
y&\longmapsto& b_{2}b_{3}\\
a&\longmapsto& a_{1}a_{2}a_{3}\\
b&\longmapsto& b_{1}b_{2}b_{3}
   \end{array}\right.
\end{align}
and 

\begin{align}\label{Psi}\Psi:
\left\{\begin{array}{ccl}
\bar{u}&\longmapsto& a_{4}\\
\bar{v}&\longmapsto& b_{4}\\
w_{2}&\longmapsto& C_{2,4}\\
w_{3}&\longmapsto& C_{3,4}\\
   \end{array}\right.&\quad&
\Psi:
\left\{\begin{array}{ccl}
u&\longmapsto& a_{3}a_{4}\\
v&\longmapsto& b_{3}b_{4}\\
w&\longmapsto& C_{2,3}C_{2,4}C^{-1}_{3,4}\\
   \end{array}\right.
&\quad&\Psi:
\left\{\begin{array}{ccl}
x&\longmapsto& a_{2}a_{3}a_{4}\\
y&\longmapsto& b_{2}b_{3}b_{4}\\
a&\longmapsto& a_{1}a_{2}a_{3}a_{4}\\
b&\longmapsto& b_{1}b_{2}b_{3}b_{4}\\
   \end{array}\right.
\end{align}

It is straightforward to prove that $\Phi$ and $\Psi$ are well defined and bijetive, by using Theorem~\ref{puras}.
\end{proof}

\begin{remark}\label{rem:sd3} The following relations are valid in $G_{3}(\mathbb T)$:
\begin{enumerate}
\item $xvx^{-1}=uvwu^{-1}$;
\item $yuy^{-1}=vuw^{-1}v^{-1}$.
\end{enumerate} 
\end{remark}

For completeness, we state below a resembling result for the Klein bottle. The isomorphism is defined in the generators in the same way as $\phi$ in~(\ref{phi}).

\begin{prop}[\cite{GP}, Remark 5.3]\label{prop:sd2K} The pure braid group $P_{2}(\mathbb K)$ is isomorphic to $G_{2}(\mathbb K)=F_{2}\rtimes (\mathbb Z\rtimes\mathbb Z)$, the semidirect product of the free group $F_{2}=\langle x,y\rangle$ with $\mathbb Z\rtimes \mathbb Z=\langle a,b\,|\,ab=ba^{-1}\rangle$, equipped with the following action
\begin{enumerate}
\begin{multicols}{2}
    \item $a^{-1}za=\left\{\begin{array}{ll}x& \mbox{if}\,\,z=x,\\
    x^{-2}y& \mbox{if}\,\,z=y;
    \end{array}\right.$

\item $b^{-1}zb=\left\{\begin{array}{ll}x^{-1}& \mbox{if}\,\,z=x,\\
    xyx& \mbox{if}\,\,z=y;
    \end{array}\right.$
    
\end{multicols}
\end{enumerate}
    
\end{prop}

\section{The action $Aut(P_n(M)) \curvearrowright S(P_n(M))$ contains certain permutations}\label{sec:aut}

In this section, we use homeomorphisms of the configuration spaces of $M\neq \mathbb S^{2}$ to obtain automorphisms of $P_n(M)$ which induce certain permutation of coordinates on the character spheres $S(P_n(M))$. This will be useful on the computation of $\Sigma^1(P_n(M))$ (Theorem~\ref{BNS-Klein}).

\begin{theorem}\label{teo:permutacoes}
Let $M\neq \mathbb S^{2}$ be a closed surface. For any $\tau\in S_{n}$, there is an automorphism of $P_{n}(M)$ whose induced automorphism on $P_{n}(M)^{Ab}$ is of the form $\rho_{i,r}\mapsto \rho_{\tau(i), r}$, for $1\leq i \leq n$ and $1\leq r \leq 2g$, if $M$ orientable (resp. $1\leq r \leq g$, if $M$ is non-orientable).  
\end{theorem}

\begin{proof}
First, notice that, since transpositions of the form $\tau=(i\ \ \ i+1)$, for $1 \leq i < n$, generate $S_n$, it suffices to show the result for such a transposition $\tau$. Given such $\tau$, let $f:F_n(M)\to F_n(M)$ be the homeomorphism
\[
f(x_1,...,x_i,x_{i+1},...,x_n)=(x_1,...,x_{i+1},x_i,...,x_n)
\] which permutes the $i$-th and ($i+1$)-th coordinates. Fix distinct base points $q_1,...,q_n \in M$ and denote $Q=(q_1,...,q_n) \in F_{n}(M)$, so that we have by definition $P_n(M) = \pi_1(F_n(M),Q)$. Since $f(Q)=(q_1,...,q_{i+1},q_i,...,q_n)=Q'$, the map $f$ induces the group isomorphism
\[
f_*:\pi_1(F_n(M),Q) \to \pi_1(F_n(M),Q')
\] with $f_*([\gamma])=[f \circ \gamma]$. Furthermore, let $\gamma*\delta$ denote the known concatenation of two paths $\gamma$ and $\delta$, meaning ``$\gamma$ followed by $\delta$''. The braid $\sigma_i \in B_n(M)$ can be naturally seen as a path $\gamma:[0,1] \to F_n(M)$ from $Q$ to $Q'$. Denote by $\hat{\gamma}:[0,1] \to F_n(M)$, $\hat{\gamma}(t)=\gamma(1-t)$ its inverse path. By basic topology, we have the group isomorphism
\begin{align*}
  \psi:\pi_1(F_n(M),Q') &\longrightarrow \pi_1(F_n(M),Q) \\
  [\delta]&\longmapsto [\gamma*\delta*\hat{\gamma}]
\end{align*} Therefore, we obtain the group automorphism $\varphi=\psi \circ f_*$ of $\pi_1(F_n(M),Q)= P_n(M)$.

Let us visualize the image under $\varphi$ of $\rho_{i,r}=[\tilde{\rho}_{i,r}]$, which is $[\gamma*(f \circ \tilde{\rho}_{i,r})*\hat{\gamma}]$. This is the homotopy class of a pure $n$-braid in $M$. See, for example, Figure~\ref{fig:relacao}. By the definitions of $\gamma$, $\tilde{\rho}_{i,r}$ and $f$, the $i$-th coordinate starts at $q_i$, passes in front of the $(i+1)$-th string via $\sigma_i$ until $q_{i+1}$, stays constant and then crosses the $(i+1)$-th string again via $\sigma_i^{-1}$. The $(i+1)$-th coordinate starts at $q_{i+1}$, passes behind the $i$-th string until $q_i$, crosses the wall $\epsilon_i$, coming back to $q_i$ and then passes behind the $i$-th string until $q_{i+1}$. All other coordinates stay constant in this composition of paths. Since $\pi_1(F_n(M),Q)=P_n(M)$, this braid is the pure braid $\sigma_i\rho_{i,r}\sigma_i^{-1}$ and, therefore, by relation (\ref{conj.rho}) and Remark~\ref{Cij}, it is
\[
\sigma_i\rho_{i,r}\sigma_i^{-1} = \sigma_i^2(\sigma_i^{-1}\rho_{i,r}\sigma_i)\sigma_i^{-2}=\sigma_i^2\sigma_i^{-2}\rho_{i+1,r}\sigma_i^{-2}=\rho_{i+1,r}\sigma_i^{-2}=\rho_{i+1,r}C_{i,i+1}^{-1}.
\] Therefore, $\varphi(\rho_{i,r})=\rho_{i+1,r}C_{i,i+1}^{-1}$. Similarly, one can see that, under the identification $\pi_1(F_n(M),Q)=P_n(M)$, the braid $\varphi(\rho_{i+1,r})$ is $\sigma_i\rho_{i+1,r}\sigma_i^{-1}$ and, therefore, by relation (\ref{conj.rho}) and Remark~\ref{Cij}, it is
\[
\sigma_i\rho_{i+1,r}\sigma_i^{-1} = \sigma_i^2(\sigma_i^{-1}\rho_{i+1,r}\sigma_i)\sigma_i^{-2}=\sigma_i^2(\rho_{i,r}\sigma_i^2)\sigma_i^{-2}=\sigma_i^2\rho_{i,r}=C_{i,i+1}\rho_{i,r}.
\]Therefore, $\varphi(\rho_{i+1,r})=C_{i,i+1}\rho_{i,r}$. Also, for $j \notin \{i,i+1\}$, again by relation (\ref{conj.rho}) we have $\varphi(\rho_{j,r})=\sigma_i\rho_{j,r}\sigma_i^{-1}=\rho_{j,r}$. Since $C_{i,i+1}\in\gamma_{2}(P_{n}(M))$ by Remark~\ref{rem:basis}, the theorem is proved for $\tau=(i\ \ \ i+1)$. This finishes our proof.
\end{proof}

The next corollary is immediate from Theorem \ref{teo:permutacoes}. From now on, we will denote the coordinates of $S(P_n(\mathbb{K}))\simeq S^{n-1}$ and $S(P_n(\mathbb{T}))\simeq S^{2n-1}$ (see Remark \ref{rem:basis}) respectively by $(x_1,...,x_n)=(\chi(b_1),...,\chi(b_n))$ and \[(y_1,...,y_n)\times(x_1,...,x_n)=(\chi(a_1),...,\chi(a_n))\times(\chi(b_1),...,\chi(b_n)).\]

\begin{cor}
For any $\tau \in S_n$, the following assertions hold:

\begin{itemize}
    \item[i)] there is an automorphism of $P_n(\mathbb{K})$ whose induced homeomorphism on $S(P_n(\mathbb{K}))$, under the identification $S(P_n(\mathbb{K}))\simeq S^{n-1}$, is the associated permutation of coordinates
\[
(x_1,x_2,...,x_n) \mapsto (x_{\tau(1)},x_{\tau(2)},...,x_{\tau(n)});
\]

    \item[ii)] there is an automorphism of $P_n(\mathbb{T})$ whose induced homeomorphism on $S(P_n(\mathbb{T}))$, under the identification $S(P_n(\mathbb{T}))\simeq S^{2n-1}$, is the associated permutation of coordinates
\[
(y_1,y_2,...,y_n)\times(x_1,x_2,...,x_n) \mapsto (y_{\tau(1)},y_{\tau(2)},...,y_{\tau(n)})\times(x_{\tau(1)},x_{\tau(2)},...,x_{\tau(n)}).
\]
\end{itemize}
\end{cor}

In particular, since $\Sigma^1$ is invariant under the sphere homeomorphisms above, we obtain the following geometric results, which will be useful on the next section for the computation of $\Sigma^1(P_n(M))$.

\begin{cor}\label{cor:sigmapermklein}
    The BNS invariant $\Sigma^1(P_n(\mathbb{K}))$ (and its complement $\Sigma^1(P_n(\mathbb{K}))^c$) is invariant under all permutation of coordinates in $S(P_n(\mathbb{K}))$.
\end{cor}

\begin{cor}\label{cor:sigmapermtoro}
    The BNS invariant $\Sigma^1(P_n(\mathbb{T}))$ (and its complement $\Sigma^1(P_n(\mathbb{T}))^c$) is invariant under all permutations in $S(P_n(\mathbb{T}))$ of the form
    \[
(y_1,y_2,...,y_n)\times(x_1,x_2,...,x_n) \mapsto (y_{\tau(1)},y_{\tau(2)},...,y_{\tau(n)})\times(x_{\tau(1)},x_{\tau(2)},...,x_{\tau(n)}).
\]
\end{cor}

\section{Computation of the BNS invariants of $B_n(M)$ and $P_n(M)$}\label{sec:computation}

In this section we compute the BNS invariants for the total and pure braid groups of some closed surfaces of low genus. On Subsection~\ref{sec:firstcases}, we compute $\Sigma^1$ for the total braid groups of $\mathbb S^{2}$, $\mathbb RP^{2}$, $\mathbb T$ and $\mathbb K$ and on Subsection~\ref{sec:esfera} for the pure braid groups of $\mathbb S^{2}$ and $\mathbb RP^{2}$. On subsection~\ref{sec:toroeklein}, we focus on the two cases that turned out to be the most difficult ones: we show Theorem \ref{BNS-Klein}, which describes $\Sigma^1$ of both $P_n(\mathbb T)$ and $P_n(\mathbb K)$.

We suppose from now on that $n \geq 2$, since the BNS invariants of the fundamental groups $\pi_1(M)=P_1(M)=B_1(M)$ of the surfaces $M$ are easy to compute. In fact, the trivial group $\pi_1(\mathbb S^2)$ and the finite group $\pi_1(\mathbb RP^2)\simeq \mathbb Z_2$ have empty character spheres and $\Sigma^1$. The groups $\pi_1(\mathbb T)\simeq \mathbb{Z}\times \mathbb{Z}$ and $\pi_1(\mathbb K)\simeq \mathbb{Z} \rtimes \mathbb{Z}$ have full BNS invariant $\Sigma^1(\pi_1(\mathbb T))=S(\pi_1(\mathbb T))\simeq \mathbb{S}^1$ and $\Sigma^1(\pi_1(\mathbb K))=S(\pi_1(\mathbb K))\simeq \mathbb{S}^0$ because they are virtually abelian, and virtually abelian groups have full BNS invariants as a direct consequence of Propositions \ref{prop:centro} and B1.11 of \cite{Strebel}.

\subsection{Computation of $\Sigma^{1}(B_{n}(M))$}\label{sec:firstcases}

Let us first obtain the BNS invariants for $B_{n}(M)$ in some cases. For the Artin braid group $B_{n}$, the computation of $\Sigma^{1}(B_{n})$ is well known. For the sake of completeness, we provide a short proof of this fact.

\begin{prop} Let $n\geq 2$. Then $\Sigma^{1}(B_{n})=S(B_{n})=\mathbb S^{0}$.
\end{prop}

\begin{proof}
Using the Artin relations, note that in the abelianized group $B_n/\gamma_2(B_n)$ we have $\sigma_i=\sigma_j$ for all $1 \leq i,j \leq n-1$, so $B_n/\gamma_2(B_n)=\left< \sigma_1 \right> \simeq \mathbb{Z}$ and, therefore, $S(B_{n})=\mathbb S^{0}$. If $[\chi] \in S(B_n)$, then $\chi(\sigma_1)\neq 0$, hence $\chi$ does not vanish the full twist $\Delta=(\sigma_{1}\cdots\sigma_{n})^{n}$, which generates the center of $B_{n}$~\cite{Chow}. Therefore, $\chi(Z(B_{n}))\neq 0$ and $\Sigma^1(B_n)=S(B_n) \simeq \mathbb{S}^0$ by Proposition \ref{prop:centro}.
\end{proof}

For the other surfaces, it is necessary to compute the character sphere of $B_{n}(M)$ first. Notice that, according to Proposition~\ref{prop:Abel}, $S(B_{n}(\mathbb RP^{2}))$ and $S(B_{n}(\mathbb S^{2}))$ are empty, and therefore their $\Sigma^{1}$ are also empty. Therefore, there is nothing to compute in these cases. If $M$ is the torus or the Klein bottle, we obtain Theorem~\ref{th:Bn}.

\begin{proof}[Proof of Theorem~\ref{th:Bn}]
By Proposition~\ref{prop:Abel}, it follows that $S(B_{n}(\mathbb T))\simeq \mathbb S^{1}$ (resp. $S(B_{n}(\mathbb K))\simeq \mathbb S^{0}$), with the torsion-free generators of $B_{n} (M)/\gamma_{2}(B_{n} (M))$ being $a,b$ if $M=\mathbb T$ (resp. $b$ if $M=\mathbb K$) by Remark~\ref{rem:a,b}. Also, if $[\chi]\in S(B_{n}(M))$, then $\chi(a_{1})=\chi(a_{i})$ and $\chi(b_{1})=\chi(b_{i})$, for all $1\leq i \leq n$. The centre of $B_{n}(M)$  is well known, namely $Z(B_{n}(\mathbb T))=\langle (a_{1}\cdots a_{n}), (b_{1}\cdots b_{n}) \rangle\simeq \mathbb Z^{2}$~\cite[Proposition 4.2]{PR} and 
$Z(B_{n}(\mathbb K))=\langle (b_{n}\cdots b_{1})^{2} \rangle\simeq \mathbb Z$~\cite[Proposition 5.2]{GP}. Now, if $[\chi] \in S(B_{n}(M))$, then $\chi(a_{1})\neq 0 $ or $\chi(b_{1})\neq 0$ if $M=\mathbb T$ (resp. $\chi(b_{1})\neq 0$ if $M=\mathbb K$); therefore, $\chi(Z(B_{n}(M)))\neq 0$ and, by Proposition~\ref{prop:centro}, it follows that $[\chi]\in\Sigma^{1}(B_{n}(M))$, as desired. 
\end{proof}

The knowledge of the generators of $Z(B_{n}(M))$ was essential in the previous proof. Since $Z(B_{n}(M))$ is trivial if $M$ is a compact surface without boundary and different from $\mathbb S^{2},\mathbb T,\mathbb RP^{2},\mathbb K$ \cite[Proposition 1.6]{PR}, the methods used above do not apply, and we were not able to compute the $\Sigma^{1}(B_{n}(M))$ for other surfaces. We intend to complete this task on a near future, as well as dealing with punctured surfaces.

\subsection{Computation of $\Sigma^{1}(P_{n}(\mathbb S^{2}))$ and $\Sigma^{1}(P_{n}(\mathbb RP^{2}))$}\label{sec:esfera}
Now we focus on pure braid groups. If $M=\mathbb RP^{2}$, then by Proposition~\ref{prop:ab}~(\ref{it:ab3}), the character sphere of $P_n(\mathbb RP^{2})$ is empty, and so $\Sigma^1(P_n(\mathbb RP^{2}))$ is empty. Let $M=\mathbb S^{2}$. We know $P_{n+1}(\mathbb S^{2})$ is finite if $0 \leq n\leq 2$~\cite{FVB}. If $n \geq 3$, then $\Sigma^{1}(P_{n+1}(\mathbb S^{2}))$ can be obtained by the knowledge of $\Sigma^{1}(P_{n})$ given in~\cite{KMM}, as we shall see in the following. Recall that by~\cite[Lemma 2.5]{KMM} the character sphere $S(P_n)$ is homeomorphic to $S^{\binom{n}{2}-1}$, and a point $[\chi]$ is determined by the images $\chi(A_{i,j})$, $1 \leq i<j \leq n$. The generators $A_{i,j}$ of $P_n$ (denoted by $S_{i,j}$ in \cite{KMM}) are written in terms of the Artin generators as

\begin{equation}\label{Aij}
A_{i,j}=\sigma_{j-1}\sigma_{j-2}...\sigma_{i+1}\sigma_i^2\sigma_{i+1}^{-1}...\sigma_{j-2}^{-1}\sigma_{j-1}^{-1}.    
\end{equation}

According to~\cite[Definition 4.3]{KMM}, a point $[\chi]$ is said to belong to a $P_3$-circle (say, $\mathcal{C}_{i,j,k})$ iff there are $1\leq i<j<k \leq n$ such that
\[
\left\{\begin{array}{ll} \chi(A_{i,j})+\chi(A_{i,k})+\chi(A_{j,k})=0,\  \mbox{and}\\
\chi(A_{r,s})=0\ \mbox{if}\ \{r,s\}\not\subset \{i,j,k\}.
    \end{array}\right.
\] A point $[\chi]$ is said to belong to a $P_4$-circle (say, $\mathcal{C}_{i,j,k,l})$ iff there are $1\leq i<j<k<l \leq n$ such that 
\[
\left\{\begin{array}{ll} \chi(A_{i,j})=\chi(A_{k,l}),\\
\chi(A_{i,k})=\chi(A_{j,l}),\\
\chi(A_{i,l})=\chi(A_{j,k}),\\ 
\chi(A_{i,j})+\chi(A_{i,k})+\chi(A_{i,l})=0,\ \mbox{and} \\
    \chi(A_{r,s})=0\ \mbox{if}\ \{r,s\}\not\subset \{i,j,k,l\}.
    \end{array}\right.
\]

\begin{theorem}\cite[Theorem A]{KMM}
The BNS-invariant for the pure braid group $P_n$ is the complement of the union of the $P_3$-circles and the $P_4$-circles in its character sphere. There are exactly $\binom{n}{3}+\binom{n}{4}$ such circles.
\end{theorem}

We use the result above and the following relation to compute $\Sigma^1(P_{n+1}(\mathbb{S}^2))$. As it is observed by~\cite[Theorem 4 (i)]{GG}, the natural epimomorphism $P_{n+1}(\mathbb S^2) \to P_3(\mathbb S^2) \simeq \mathbb{Z}_2$ has a kernel isomorphic to \[H_{n-2} \simeq P_{n-2}(\mathbb S^{2}\setminus\{x_{0},x_{1},x_{2}\}) \simeq P_{n-2}(\mathbb D^{2}\setminus\{x_{1},x_{2}\}).\] Furthermore, this exact sequence gives rise to the isomorphism $P_{n+1}(\mathbb S^{2}) \simeq H_{n-2}\times \mathbb Z_{2}$. We will therefore use the identifications $P_{n+1}(\mathbb S^{2}) =H_{n-2} \times \mathbb Z_{2}$ and $H_{n-2}=P_{n-2}(\mathbb D^{2}\setminus\{x_{1},x_{2}\})$. 

Now, we make use of~\cite[Theorem 4 (ii)]{GG}: the natural epimomorphism $P_n \to P_2 =\left< \Delta \right>\simeq \mathbb{Z}$ also has a kernel isomorphic to $H_{n-2}$. Therefore, the generators of $H_{n-2}$ can be identified with the Artin generators of $P_n$ (except $A_{1,2}$) and will be denoted by $\tilde{A}_{i,j}$, for $1 \leq i <j \leq n$, $\{i,j\}\neq \{1,2\}$. In addition, this exact sequence gives rise to an isomorphism $\varphi:P_n\to H_{n-2}\times \mathbb{Z}$, by identifying $\Delta$ with the full twist $\Delta=(\sigma_{1}\cdots\sigma_{n})^{n} \in P_n$, which is a pure braid and can also be written in term of the generators of $P_{n}$ as $\Delta=A_{1,2}(A_{1,3}A_{2,3})...(A_{1,n}...A_{n-1,n})$, as one may check by using equation \ref{Aij}. If $\omega \in H_{n-2}$ is defined by $\omega^{-1}=(\tilde{A}_{1,3}\tilde{A}_{2,3})...(\tilde{A}_{1,n}...\tilde{A}_{n-1,n})$, then by using combinatorial notation we have
\[
\left\{\begin{array}{ll} \varphi(A_{1,2})=\Delta\cdot\omega,\ \mbox{and}\\
\varphi(A_{i,j})=\tilde{A}_{i,j}\ \mbox{if}\ \{i,j\}\neq \{1,2\}.
    \end{array}\right.
\]

\begin{defi}\label{def:circles} We say a point $[\chi] \in S(H_{n-2})$ belongs to a $P_3$-circle $\widetilde{C}_{i,j,k}$ for some $1 \leq i<j<k \leq n$ iff there are two numbers $(p,q)\neq (0,0)$ such that the following equations are valid:
\[
\widetilde{C}_{i,j,k}:\left\{\begin{array}{l}\mbox{if}\,\,\{i,j\}=\{1,2\}:\displaystyle\left\{\begin{array}{lr} \chi(\tilde{A}_{i,k})=p,\\
\chi(\tilde{A}_{j,k})=q,\\
\chi(\tilde{A}_{r,s})=0,\ \mbox{for}\ \{r,s\}\not\subset \{i,j,k\}
\end{array}\right.
\\
\mbox{if}\,\,\{i,j\}\neq\{1,2\}:\displaystyle\left\{\begin{array}{lr} \chi(\tilde{A}_{i,k})=p,\\
\chi(\tilde{A}_{j,k})=q,\\
\chi(\tilde{A}_{i,j})=-(p+q),\\
\chi(\tilde{A}_{r,s})=0,\ \mbox{for}\ \{1,2\}\neq\{r,s\}\not\subset \{i,j,k\}
\end{array}\right. 

\end{array}\right.
\]

Similarly, we say a point $[\chi] \in S(H_{n-2})$ belongs to a $P_4$-circle $\widetilde{C}_{i,j,k,l}$ for some $1 \leq i<j<k<l \leq n$ iff there are two numbers $(p,q)\neq (0,0)$ such that the following equations are valid:
\[
\widetilde{C}_{i,j,k,l}:\left\{\begin{array}{l}\mbox{if}\,\,\{i,j\}=\{1,2\}:\displaystyle\left\{\begin{array}{lr} \chi(\tilde{A}_{i,k})=\chi(\tilde{A}_{j,l})=p,\\
\chi(\tilde{A}_{i,l})=\chi(\tilde{A}_{j,k})=q,\\
\chi(\tilde{A}_{k,l})=-(p+q),\\
\chi(\tilde{A}_{r,s})=0,\ \mbox{for}\ \{r,s\}\not\subset \{i,j,k,l\}
\end{array}\right. 
\\
\mbox{if}\,\,\{i,j\}\neq\{1,2\}: \displaystyle\left\{\begin{array}{lr} \chi(\tilde{A}_{i,k})=\chi(\tilde{A}_{j,l})=p,\\
\chi(\tilde{A}_{i,l})=\chi(\tilde{A}_{j,k})=q,\\
\chi(\tilde{A}_{i,j})=\chi(\tilde{A}_{k,l})=-(p+q),\\
\chi(\tilde{A}_{r,s})=0,\ \{1,2\}\neq\{r,s\}\not\subset \{i,j,k,l\}
\end{array}\right. 
\end{array}\right.
\]
\end{defi}

\begin{proof}[Proof of Theorem~\ref{th:PnS2}]

Since $\mathbb{Z}_2$ is finite, $S(\mathbb Z_{2})$ must be empty and the map $\pi_1^*$ of Proposition \ref{sigmacart} is a homeomorphism. Then $\Sigma^{1}(P_{n+1}(\mathbb S^{2}))^c = \pi_1^*(\Sigma^1(H_{n-2})^c) \cup \pi_2^*(\Sigma^1(\mathbb{Z}_2)^c)=\Sigma^1(H_{n-2})^c$. First, let $n=3$. Then, $H_{n-2}=\pi_1(\mathbb D^{2}\setminus\{x_{1},x_{2}\}) \simeq F_2$ is free. Since finitely generated free groups have empty $\Sigma$-invariant (\cite[Section A2.1a, Example 3]{Strebel}), we have $\Sigma^{1}(P_{n+1}(\mathbb S^{2}))=\emptyset$. Now, let $n \geq 4$. Since a $\Sigma^1$ invariant is mapped bijectively onto a $\Sigma^1$ invariant under the induced map of a group isomorphism (see Section B1.2a of \cite{Strebel}), we have
\[
\Sigma^1(H_{n-2}\times \mathbb{Z})^c =(\varphi^{-1})^*(\Sigma^1(P_n)^c)\\
=\{[\chi \circ \varphi^{-1}]\ |\ [\chi] \in \Sigma^1(P_n)^c\}.
\]
Now, since $\chi \circ \varphi^{-1}(\tilde{A}_{i,j})=\chi(A_{i,j})$ for all generators $\tilde{A}_{i,j}$ of $H_{n-2}$, it follows that the image under $(\varphi^{-1})^*$ of a $P_3$-circle $\mathcal{C}_{i,j,k}$ (respectively, of a $P_4$-circle $\mathcal{C}_{i,j,k,l}$) of $S(P_n)$ is simply obtained by deleting the first coordinate $\chi(A_{1,2})$ of the circle $\mathcal{C}_{i,j,k}$ (resp. $\mathcal{C}_{i,j,k,l}$) - whether it is zero or not - and adding a zero last coordinate $\chi \circ \varphi^{-1}(\Delta)=\chi(\Delta)=0$. So, $\Sigma^1(H_{n-2}\times \mathbb{Z})^c$ is the union of these new circles. Finally, again by Proposition \ref{sigmacart} we have
\[\Sigma^1(H_{n-2}\times \mathbb{Z})^c = \pi_1^*(\Sigma^1(H_{n-2})^c) \cup \pi_2^*(\Sigma^1(\mathbb{Z})^c)=\pi_1^*(\Sigma^1(H_{n-2})^c),\]
which means $\Sigma^1(H_{n-2})^c$ is obtained by deleting the last coordinate $\chi(\Delta)=0$ of all the circles of $\Sigma^1(H_{n-2}\times \mathbb{Z})^c$. If we do this, it is easy to see that we obtain exactly the $P_3$-circles and $P_4$-circles of Definition \ref{def:circles}. It is straightforward to check that all these circles are pairwise disjoint. This completes our proof.
\end{proof}

\subsection{Computation of $\Sigma^{1}(P_{n}(\mathbb T))$ and $\Sigma^{1}(P_{n}(\mathbb K))$}\label{sec:toroeklein}

This subsection is dedicated to present the BNS invariant for the pure braid groups of the torus and the Klein bottle by proving Theorem \ref{BNS-Klein}. The proof will be by induction on $n$, for $n\geq2$, and for the sake of shortness, we will simultaneously deal with both cases $M=\mathbb T$ and $M=\mathbb K$, for both inductions turned out to have very similar aspects. In order to start the induction process, we first deal with the case $n=2$, as one can see in the following. The case of the torus is immediate since $P_{2}(\mathbb T)$ is a direct product; nonetheless, the case of the Klein bottle requires some more effort.

\begin{lema}\label{BNS.P2} For $(p,q) \in\mathbb R^2\setminus\{(0,0)\}$, define $[\chi_{p,q}]\in S(P_{2}(\mathbb T))$ by $\chi_{p,q}(a_{1})=\chi_{p,q}(a^{-1}_{2})=p$ and $\chi_{p,q}(b_{1})=\chi_{p,q}(b^{-1}_{2})=q$. Then,
$$\Sigma^{1}(P_{2}(\mathbb T)))^{c}=\left\{[\chi_{p,q}]\,|\,(p,q)\in\mathbb R^2\setminus\{(0,0)\}\right\}\cong\mathbb S^{1}.$$
\end{lema}

\begin{proof}
The proof follows directly from Proposition~\ref{prop:Pn}~\ref{it:P2} and Theorem~\ref{sigmacart}.
\end{proof}

\begin{lema}\label{lem:P2}
Define $[\chi] \in S(P_2(\mathbb{K}))$ by $\chi(b_1)=\chi(b_2^{-1})=1$. Then, \[\Sigma^1(P_2(\mathbb{K}))^c=\{[\chi],[-\chi]\}.\]
\end{lema}

\begin{proof}
First of all, since an isomorphism of groups induces a natural homeomorphism between the corresponding character spheres \cite[Section B1.2a]{Strebel}, we now consider $P_{2}(\mathbb K)$ as the group $G_{2}(\mathbb K)$ with presentation given by Proposition~\ref{prop:sd2K}. We choose the set of generators $X=\{x,y,a,b\}$ for the Cayley graph of $P_{2}(\mathbb K)$. By the isomorphism $P_{2}(\mathbb K) \simeq G_{2}(\mathbb K)$ of Equation (\ref{phi}), we must have $\chi(x)=0=\chi(a)$, $\chi(y)=-1$ and, since $b=b_2b_1$ (under the identification), we have $\chi(b)=\chi(b_2b_1)=\chi(b_2)+\chi(b_1)=-1+1=0$. Hence, for any vertex $g$ in $\Gamma(P_{2}(\mathbb K),X)$, the number $-\chi(g)$ is the sum of the powers of $y$ in $g$. Furthermore, the points $[\chi],[-\chi]$ are the only ones in $S(P_2(\mathbb K))$ that vanish the center $Z(P_2(\mathbb K))=\left< b^2 \right>$~\cite[Proposition 5.2]{GP}. Then, it follows directly from Proposition \ref{prop:centro} that $\Sigma^1(P_2(\mathbb{K}))^c \subset \{[\chi],[-\chi]\}$.

Let us show $\{[\chi],[-\chi]\} \subset \Sigma^1(P_2(\mathbb{K}))^c$. Suppose first by contradiction that $[\chi] \in \Sigma^1(P_{2}(\mathbb K))$. Then, in particular, there is a path $p$ from the vertex $1$ to $yxy^{-1}$ inside $\Gamma_\chi$. From now on, we will use $p$ to construct a path on the Cayley graph of the free group $F_2$ which cannot exist because $\Sigma^{1}(F_{2})=\emptyset$ (\cite{Strebel}, A2.1a, item (3)). We will also use a straightforward normal form for $P_{2}(\mathbb K)$, which comes from its semidirect product structures: any element $g \in P_{2}(\mathbb K)$ can be uniquely written as $g=\omega a^nb^m$, for $\omega \in F_2=F(x,y)$ and $n,m \in \mathbb{Z}$. Note that $\chi(g)=\chi(\omega)$. For every such $g$, let us describe the normal form of $gz$, $z \in X^{\pm 1}$, on the right side of equations below. One can straightforwardly check that
\begin{multicols}{2}
\begin{enumerate}
    \item $\omega a^nb^ma=\omega a^{n+(-1)^m}b^m$;
    \item $\omega a^nb^mb=\omega a^nb^{m+1}$;
    \item $\omega a^nb^mx=\omega x^{(-1)^m}a^nb^m$;
    \item $\displaystyle \omega a^nb^my=\left\{\begin{array}{lr} \omega x^{2n}ya^nb^m, & m\ \mbox{even}\\ 
\omega x^{2n+1}yxa^nb^m, & m\ \mbox{odd} \end{array}\right.$
\columnbreak
   \item $\omega a^nb^ma^{-1}=\omega a^{n+(-1)^{m+1}}b^m$;
    \item $\omega a^nb^mb^{-1}=\omega a^nb^{m-1}$;
    \item $\omega a^nb^mx^{-1}=\omega x^{(-1)^{m+1}}a^nb^m$;    
    \item $\displaystyle \omega a^nb^my^{-1}=\left\{\begin{array}{lr} \omega y^{-1}x^{-2n}a^nb^m,\, & m\ \mbox{even}\\ 
\omega x^{-1}y^{-1}x^{-2n-1}a^nb^m,\,& m\ \mbox{odd} \end{array}\right.$
\end{enumerate}
\end{multicols}
The path $p$ provides us with a sequence $g_0,g_1,...,g_k \in P_{2}(\mathbb K)$ such that $g_0=1$, $g_k=yxy^{-1}$ and $g_{i+1}=g_iz_i$ for some $z_i \in X^{\pm 1}$, for every $0 \leq i < k$. Furthermore, $p$ to be inside $\Gamma_\chi$ gives us $\chi(g_i) \geq 0$ for $0 \leq i \leq k$. By writing $g_i=\omega_ia^{n_i}b^{m_i}$ for $1 \leq i$ we get $\chi(\omega_i)=\chi(g_i) \geq 0$. By uniqueness of the normal form, we have a sequence $\omega_0,\omega_1,...,\omega_k \in F_2$ with $\omega_0=1$, $\omega_k=yxy^{-1}$ and $\chi(\omega_i)\geq 0$. For us to obtain a path on the Cayley subgraph $\Gamma(F_2,\{x,y\})_{\chi|_{F_2}}$, we will connect, for every $i\geq 0$, the vertices $\omega_i$ and $\omega_{i+1}$ inside this subgraph. Denote for a moment $\omega_i=\omega$, $\omega_{i+1}=\omega'$ and $z_i=z$. Since $g_{i+1}=g_iz$, $z \in X^{\pm 1}$, there are $8$ possibilities for $z$. For each of them, we will use one of the $8$ normal forms above to create a path $q$ from $\omega$ to $\omega'$ inside $\Gamma(F_2)_\chi = \Gamma(F_2,\{x,y\})_{\chi|_{F_2}}$. One can then easily check that equation $\chi(x)=0$ guarantees that  $\nu_\chi(q)\geq 0$, so the paths $q$ below are inside $\Gamma(F_2)_\chi$.
\begin{itemize}
    \item Case $z \in \{a,a^{-1},b,b^{-1}\}$: in this case, $\omega'=\omega$ and $q$ is the constant (or trivial) path;

    \item Case $z=x$ (resp. $z=x^{-1}$): in this case, $\omega'=\omega x^{(-1)^m}$ (resp. $\omega'=\omega x^{(-1)^{m+1}}$); therefore, the path $q=(\omega,x^{(-1)^m})$ (resp, $q=(\omega,x^{(-1)^{m+1}})$) connects $\omega$ to $\omega'$ inside $\Gamma(F_2)_\chi$;

    \item Case $z=y$, $m$ even (resp. $m$ odd): in this case, $\omega'=\omega x^{2n}y$ (resp. $\omega'=\omega x^{2n+1}yx$). If $n \geq 0$, the path $q=(\omega,x...xy)$ (resp. $q=(\omega,x...xyx)$), with direction $x$ being travelled $2n$ (resp $2n+1$) times, connects $\omega$ to $\omega'$ inside  $\Gamma(F_2)_\chi$. If $n < 0$, $q$ can be taken similarly as $q=(\omega,x^{-1}...x^{-1}y)$ (resp. $q=(\omega,x^{-1}...x^{-1}yx)$);

    \item Case $z=y^{-1}$, $m$ even (resp. $m$ odd): in this case, $\omega'=\omega y^{-1}x^{-2n}$ (resp. $\omega'=\omega x^{-1}y^{-1}x^{-2n-1}$). If $n < 0$, the path $q$ can be taken as $q=(\omega,y^{-1}x...x)$ (resp. $q=(\omega,x^{-1}y^{-1}x...x)$), with direction $x$ being travelled $-2n$ (resp. $-2n-1$) times. If $n \geq 0$, $q$ can be similarly taken as $q=(\omega,y^{-1}x^{-1}...x^{-1})$ (resp $q=(\omega,x^{-1}y^{-1}x^{-1}...x^{-1})$).
\end{itemize}
Thus, we obtain a path $p'$ from $1$ to $yxy^{-1}$ in $\Gamma(F_2)_\chi$. By removing possible backtrackings of the form $zz^{-1}$ in $p'$, we can also assume $p'$ to be a geodesic. This is a contradiction, for $\Gamma(F_2)$ is known to be a tree and the only geodesic from $1$ to $yxy^{-1}$ is the path $(1,yxy^{-1})$, which is not inside $\Gamma(F_2)_\chi$, for $\chi(y)=-1<0$. It follows that $[\chi] \notin \Sigma^1(P_2(\mathbb{K}))$. In a similar way, by using the element $y^{-1}xy$, one shows that $[-\chi] \notin \Sigma^1(P_2(\mathbb{K}))$.
\end{proof}

\begin{remark}\label{rem:shortproof}
We could have obtained a much shorter proof of the fact that $\{[\chi],[-\chi]\} \subset \Sigma^1(P_2(\mathbb{K}))^c$ on Lemma \ref{lem:P2} by using \cite{GP}, as follows: from ~\cite[Theorem 5.4]{GP} one can conclude that $\gamma_2(P_2(\mathbb{K}))=(P_2(\mathbb{K}))'$ is not finitely generated. By \cite[Theorem A4.1]{Strebel}, either $[\chi]$ or $[-\chi]$ are outside $\Sigma^1$. Then, by Corollary \ref{cor:sigmapermklein}, we conclude that $\{[\chi],[-\chi]\} \subset \Sigma^1(P_2(\mathbb{K}))^c$, as desired. However, with \cite[Theorem A4.1]{Strebel}, and our proof of Lemma \ref{lem:P2} we get a new and independent proof for the fact that the commutator subgroup $(P_2(\mathbb{K}))'$ of $P_2(\mathbb{K})$ is not finitely generated.
\end{remark}

\begin{cor}\label{cor:complementoBNSKlein} Let $M$ be the torus or the Klein bottle, $n\geq 3$ and $[\chi]\in S(P_{n}(M))$. If $\chi(a_{i})=\chi(a^{-1}_{j})=p$ and $\chi(b_{i})=\chi(b^{-1}_{j})=q$, for $M=\mathbb T$ (resp. $\chi(b_{i})=\chi(b^{-1}_{j})=1$, for $M=\mathbb K$) and $\chi(a_{k})=\chi(b_{k})=0$ for  $1\leq k \leq n$, $k\neq i,j$, then $[\chi]\notin\Sigma^{1}(P_{n}(M))$.
\end{cor}

\begin{proof}
Consider the projection $\beta_{i,j}:F_{n}(M)\to F_{2}(M)$, $(x_{1},\ldots,x_{i},\ldots,x_{j},\ldots,x_{n})\mapsto (x_{i},x_{j})$, and the induced homomorphism ${\beta_{i,j}}_{\ast}:P_{n}(M)\to P_{2}(M)$, which is a composition of some of the homomorphisms from the Fadell-Neuwirth short exact sequence~(\ref{seq:FN}).  Geometrically, each $n$-braid is sent to a $2$-braid, by deleting all but the $i$-th and $j$-th strings. The result follows then directly from \cite[Corollary B1.8]{Strebel} and Lemmas~\ref{BNS.P2} and~\ref{lem:P2}.
\end{proof}

It turns out that, since the sphere $S(P_n(\mathbb T))$ has higher dimension than $S(P_n(\mathbb K))$, the induction step for the case $M=\mathbb T$ needed some additional lemmas about the particular cases $n=3,4$ to work. We deal with them in what follows.

\begin{lema}\label{BNS.P3} Let $p,q>0$ and $[\chi]\in S(P_{3}(\mathbb T))$, such that  
\begin{align*}
\chi(b_{i})=\chi(b^{-1}_{j})=q, \quad
\chi(a_{\tau(i)})=\chi(a^{-1}_{\tau(j)})=p,\quad \chi(b_{k})=\chi(a_{\tau(k)})=0,
\end{align*} with $\tau\in S_{3}$, $\tau(k)\neq k$ and $\{i,j,k\}=\{1,2,3\}$. Then, $[\chi]$ belongs to $\Sigma^{1}(P_{3}(\mathbb T))$.
\end{lema}

\begin{proof}
First of all, by Corollary~\ref{cor:sigmapermtoro}, we can put the coordinates $(\chi(a_{1}),\chi(a_{2}),\chi(a_{3}))$ in ascending order and reduce the possibilities and consider only the cases where  $(\chi(a_{1}),\chi(a_{2}),\chi(a_{3}))\times(\chi(b_{1}),\chi(b_{2}),\chi(b_{3}))$ equals one of the following:
\begin{multicols}{2}
\begin{enumerate}[label=(\alph*)]
\item\label{it:t1} $(-p,0,p)\times (-q,q,0)$;
\item\label{it:t2} $(-p,0,p)\times (q,-q,0)$;
\item\label{it:t3} $(-p,0,p)\times (0,-q,q)$;
\item\label{it:t4} $(-p,0,p)\times (0,q,-q)$.
\end{enumerate}   
\end{multicols}
Also, we will consider the isomorphism $\Phi$ between the group $G_{3}(\mathbb T)$ and $P_{3}(\mathbb T)$ given in~(\ref{Phi}), in the proof of Proposition~\ref{prop:Pn}~\ref{it:P3}, and use the Geometric Criterion given in Theorem~\ref{geometriccrit} to show that $[\chi]\in\Sigma^{1}(P_{3}(\mathbb T))$, by constructing the paths $p_z$ satisfying $\nu_\chi(p_z)>\nu_\chi(1,z)$, for $z \in Z=\left\{a^{\pm},b^{\pm},x^{\pm},y^{\pm},u^{\pm},v^{\pm},w^{\pm}\right\}$, in each of the cases above.

To prove~\ref{it:t1} (resp.~\ref{it:t2}), notice that the only generators of $G_{3}(\mathbb T)$ that have a non-null image are $x,u,y$, with
\begin{align*}
\chi(x)=\chi(u)=p>0\quad\mbox{and}\quad \chi(y)=q>0\quad (\mbox{resp.}\, \chi(y)=-q<0).
\end{align*}
Fix $t=x$. If $z\in\left\{a^{\pm},b^{\pm},u^{\pm},w^{\pm}\right\}$ then $t$ commute with $z$, and we can choose the trivial path $p_z=(t,z)$. To construct the paths $p_z$ for $z\in\left\{y^{\pm},v^{\pm}\right\}$, we will use the relations in Proposition~\ref{prop:Pn}~\ref{it:P3} and Remark~\ref{rem:sd3}. It is straightforward to check that the following relations are valid in $G_{3}(\mathbb T)$, by writing the elements on the right side on their normal forms, coming from the semidirect product structure. For example, the normal form of $x(w^{1}v^{-1}yu^{-1}vuvy^{-1})$ is $vx$.
\begin{align}
\label{it.v1}x^{-1}vx&=w^{1}v^{-1}yu^{-1}vuvy^{-1},\\
\label{it.y1}x^{-1}yx&=ux^{-1}yw^{-1}v^{-1}u^{-1}vx,\\
\label{it.v2}x^{-1}vx&=vwy^{-1}u^{-1}vuyw^{-1}v^{-1}w^{-1},\\
\label{it.y2}x^{-1}yx&=x^{-1}vuyw^{-1}v^{-1}xu^{-1}.
\end{align}
Relations~(\ref{it.v1}) and~(\ref{it.y1}) (resp.~(\ref{it.v2}) and~(\ref{it.y2})) give us the paths $p_{v}$ and $p_{y}$ to prove~\ref{it:t1} (resp.~\ref{it:t2}). By taking the inverse of each relation, we obtain the paths $p_{v^{-1}}$ and $p_{y^{-1}}$.

To prove~\ref{it:t3} (resp.~\ref{it:t4}), notice that the only generators that have a non-null image are $x,u,v$, with
\begin{align*}
\chi(x)=\chi(u)=p>0\quad\mbox{and}\quad \chi(v)=q>0\quad (\mbox{resp.}\, \chi(v)=-q<0), 
\end{align*}
This time we fix $t=v$ (resp. $t=v^{-1}$) to obtain the paths $p_z$. Now, if $z\in\left\{a^{\pm},b^{\pm},y^{\pm}\right\}$ then $t$ commutes with $z$ and we can choose the trivial path $p_z=(t,z)$. To construct the paths $p_z$ for $z\in\left\{u^{\pm},x^{\pm},w^{\pm}\right\}$, we will use again the relations in Proposition~\ref{prop:Pn}~\ref{it:P3} and Remark~\ref{rem:sd3}. It is straightforward to check that the following relations are valid in $G_{3}(\mathbb T)$, by writing the elements on the right side on their normal forms, coming from the semidirect product structure.
\begin{align}
\label{it.u1}v^{-1}uv&=y^{-1}xv^{-1}ux^{-1}vy,\\
\label{it.x1}v^{-1}xv&=wxu^{-1}y^{-1}uw^{-1}y,\\
\label{it.w1}v^{-1}wv&=uw^{-1}v^{-1}yxu^{-1}wvx^{-1}y^{-1}vuw^{-1}v^{-1}u^{-1},\\
\label{it.u2}vuv^{-1}&=yxvux^{-1}y^{-1}v^{-1},\\
\label{it.x2}vxv^{-1}&=xu^{-1}yuy^{-1},\\
\label{it.w2}vwv^{-1}&=uvwy^{-1}xu^{-1}wv^{-1}x^{-1}yv^{-1}uvwu^{-1}.
\end{align}
Relations~(\ref{it.u1}),~(\ref{it.x1}) and~(\ref{it.w1}) (resp.~(\ref{it.u2}),~(\ref{it.x2}) and~(\ref{it.w2})) give us the paths $p_{u}$, $p_{x}$ and $p_{w}$ to prove~\ref{it:t3} (resp.~\ref{it:t4}). By taking the inverse of each relation, we similarly obtain the remaining paths $p_{u^{-1}}$, $p_{x^{-1}}$ and $p_{w^{-1}}$.
\end{proof}

\begin{lema}\label{BNS.P4} 
Let $p, q> 0$ and $[\chi]\in S(P_{4}(\mathbb T))$. Then, $[\chi]$ belongs to $\Sigma^{1}(P_{4}(\mathbb T))$ if $$(\chi(a_{1}),\chi(a_{2}),\chi(a_{3}), \chi(a_{4}))\times(\chi(b_{1}),\chi(b_{2}),\chi(b_{3}),\chi(b_{4}))$$ is equal to one the following cases:
\begin{enumerate}[label=(\alph*)]
\item\label{it:s1} $(-p,0,0,p)\times (0,-q,q,0)$;
\item\label{it:s2} $(-p,0,0,p)\times (0,q,-q,0)$.
\end{enumerate}
\end{lema}

\begin{proof}
We will use the Geometric Criterion, given in Theorem~\ref{geometriccrit} and the isomorphism $\Psi$ between $G_{4}(\mathbb T)$ and $P_{4}(\mathbb T)$ given in~(\ref{Psi}), in the proof of Proposition~\ref{prop:Pn}~\ref{it:P4}. To show that $[\chi]\in\Sigma^{1}(P_{4}(\mathbb T))$ in case~\ref{it:s1} (resp. case~\ref{it:s2}) we fix $t=v$ (resp. $t=v^{-1})$. Notice that the only generators of $G_{4}(\mathbb T)$ that have a non-null image are $x,u,\bar{u},v$, with
\begin{align*}
\chi(x)=\chi(u)=\chi(\bar{u})=p>0\quad\mbox{and}\quad \chi(v)=q>0\quad (\mbox{resp.}\, \chi(v)=-q<0),
\end{align*}
so we use the same path from case~\ref{it:t3} (resp. case~\ref{it:t4}) from the proof of Lemma~\ref{BNS.P3} if $z\in\{a^{\pm},b^{\pm},x^{\pm},y^{\pm},u^{\pm},v^{\pm},w^{\pm}\}$. Relation~\ref{it:f} in Proposition~\ref{prop:Pn}~\ref{it:P4} gives us the remaining path $p_z$ for $z\in\{\bar{u}^{\pm},\bar{v}^{\pm},w_{2}^{\pm},w_{3}^{\pm}\}$, since $\chi(\bar{v})=\chi(w_{2})=\chi(w_{3})=0$.    
\end{proof}

Now, we are finally able to show Theorem~\ref{BNS-Klein}.

\begin{proof}[Proof of Theorem~\ref{BNS-Klein}]

We prove by induction on $n\geq 2$. If $n=2$, the theorem is valid by Lemmas~\ref{BNS.P2} and~\ref{lem:P2}. Let $n\geq 3$ and suppose, therefore, that the theorem is true for $n-1$. Denote by $A_{n}(M)$ the following set
\begin{align*}
A_n(M)=\left\{\begin{array}{ll} \{[\chi_{i,j,p,q}]\ |\ 1 \leq i,j \leq n,\, i\neq j,\,(p,q)\neq (0,0)\},&M=\mathbb T;\\
\{[\chi_{i,j}]\ |\ 1 \leq i,j \leq n,\, i\neq j\},&M=\mathbb K.\end{array}\right.    
\end{align*}
It follows from Corollary \ref{cor:complementoBNSKlein} that $A_n(M) \subset \Sigma^{1}(P_{n}(M))^c$. We are then left to show that $\Sigma^{1}(P_{n}(M))^c \subset A_n(M)$. Consider $[\chi] \notin A_n(M)$ and let us show that $[\chi] \in \Sigma^{1}(P_{n}(M))$. Note that, if $[\chi]\in S(P_{n}(M))$ then $\chi(C_{i,j})=0$ for all $i,j$, by Remark~\ref{rem:basis}. If $M=\mathbb K$, we also have that $\chi(a_{i})=0$ for $i=1,\ldots n$, by Remark~\ref{rem:basis}.  

In the following, we will use that $Z(P_{n}(M))$ is generated by $(a_{1}\cdots a_{n})$ and $(b_{1}\cdots b_{n})$, if $M=\mathbb T$~\cite[Propostion 4.2]{PR} (resp. by $(b_{n}\cdots b_{1})^{2}$, if $M=\mathbb K$~\cite[Proposition 5.2]{GP}). If $[\chi]\in S(P_{n}(M))$ is such that $\sum^{n}_{i=1}\chi(a_{i})\neq 0$ or $\sum^{n}_{i=1}\chi(b_{i})\neq 0$, it follows that $[\chi]\in \Sigma^{1}(P_{n}(M))$ by Proposition~\ref{prop:centro}. Suppose, from now on, that 
\begin{align}\label{somazero}
\sum^{n}_{i=1}\chi(a_{i})= 0\quad\mbox{and}\quad\sum^{n}_{i=1}\chi(b_{i})= 0.    
\end{align} 

Thanks to Corollaries~\ref{cor:sigmapermklein} and~\ref{cor:sigmapermtoro}, we can assume without loss of generality that
\begin{align}\label{ordem}
\chi(b_{1})\leq \chi(b_{2})\leq\cdots\leq \chi(b_{n}).
\end{align}

We will analyze the following cases (which cover all possibilities) and show that $[\chi] \in \Sigma^1(P_n(M))$ on each one. Notice that cases~\ref{caso1} and~\ref{caso2} can only occur if $M=\mathbb T$ and, in cases~\ref{casoii} and~\ref{casoiii}, we have $\chi(b_{n})>0$ and $\chi(b_{1})<0$. 

\begin{enumerate}[label=(\roman*)]
\item\label{casoi} $\chi(b_{j})=\chi(a_{j})=0$, for some $1<j<n$; 

\item\label{casoii} $|\chi(b_{1})|>|\chi(b_{n})|$;
    
\item\label{casoiii} $|\chi(b_{1})|<|\chi(b_{n})|$;

\item\label{caso1} $|\chi(b_{n})|=|\chi(b_{1})|$ and there exists $1 \leq j \leq n$ such that $\chi(a_{j})+ \chi(a_{k})<0$, for $1\leq k\leq n$;

\item\label{caso2} $|\chi(b_{n})|=|\chi(b_{1})|$ and there exists $1 \leq j \leq n$ such that $\chi(a_{j})+ \chi(a_{k})>0$, for $1\leq k\leq n$;

\item\label{casoiv} $|\chi(b_{n})|=|\chi(b_{1})|$ and all cases above do not hold.
\end{enumerate}

On case~\ref{casoi}, consider the induced homomorphim ${p_{j}}_{\ast}:P_{n}(M) \rightarrow P_{n-1}(M)$ of the projection defined in~(\ref{rhoj}) of the Fadell-Neuwirth short exact sequence~(\ref{seq:FN}). We have $[\chi]=[\Tilde{\chi}\circ {p_{j}}_{\ast}]$ for some $[\Tilde{\chi}] \in S(P_{n-1}(M))$. Since $[\chi] \notin A_n(M)$ we have $[\Tilde{\chi}] \notin A_{n-1}(M)$, so $[\Tilde{\chi}] \in \Sigma^{1}(P_{n-1}(M))$ by the induction hypothesis. Since the kernel $\pi_{1}(M\setminus\left\{n-1\,\,\mbox{pts}\right\})$ is finitely generated, it follows from \cite[Corollary B1.8]{Strebel} that $[\chi]\in \Sigma^{1}(P_{n}(M))$.

On case~\ref{casoii}, first notice that, by~(\ref{ordem}), we have
\begin{equation*}\chi(b^{-1}_{1})+\chi(b^{-1}_{j})>0,\,\,\mbox{for all}\,\, 2\leq j\leq n.
\end{equation*}
Let us use the Geometric Criterion, given in Theorem~\ref{geometriccrit}, to show that $[\chi] \in \Sigma^{1}(P_{n}(M))$. Since $\chi(b^{-1}_{1})>0$, fix $t=b_1^{-1}$  and remember that, by Theorem~\ref{puras},
\[
Z=\left\{a^{\pm1}_{i},b^{\pm1}_{i}, C^{\pm1}_{j,k}\,|\,1\leq i\leq n,\ 1\leq j<k \leq n\right\}
\]
is a set of generators for $P_{n}(M)$ with their inverses. If $z=a_{j}$ with $1<j\leq n$, we construct the path $p_z$ satisfying Theorem \ref{geometriccrit} as follows: by Proposition~\ref{prop:add}, relation~\ref{it:add2} with $i=1$ we have $b_{1}a_{j}b^{-1}_{1}=a_{j}C^{-1}_{1,j}C_{2,j}$. Operating by $b_1^{-1}$ on the left we obtain $a_{j}b^{-1}_{1}=b_{1}^{-1}a_{j}C^{-1}_{1,j}C_{2,j}$, which implies the path $p_z=(b_{1}^{-1},a_{j}C^{-1}_{1,j}C_{2,j})$ satisfies Theorem \ref{geometriccrit}, for \[\nu_\chi(p_z)=\chi(b_1^{-1})+\chi(a_{j})>\chi(a_{j})=\nu_\chi(1,a_j).\]

Similarly, for $z=a_{j}^{-1}$, we do the following: again by Proposition~\ref{prop:add}, relation~\ref{it:add2} with $i=1$ we have $b_{1}a_{j}^{-1}b^{-1}_{1}=C_{2,j}^{-1}C_{1,j}a_{j}^{-1}$, which implies $a_{j}^{-1}b^{-1}_{1}=b_{1}^{-1}C_{2,j}^{-1}C_{1,j}a_{j}^{-1}$. Then, the path $p_z=(b_{1}^{-1},C_{2,j}^{-1}C_{1,j}a_{j}^{-1})$ satisfies Theorem \ref{geometriccrit}, for $\nu_\chi(p_z)=\chi(b_1^{-1})+\chi(a^{-1}_{j})>\chi(a^{-1}_{j})=\nu_\chi(1,a_j^{-1})$. With the same strategy as above, for $z=C^{\pm1}_{1,j}$, $1<j\leq n$, we can easily obtain the desired path $p_{z}$ satisfying Theorem \ref{geometriccrit} by using relation~\ref{it:add4}. The fact $\chi(b^{-1}_{1})+\chi(b^{-1}_{j})>0$ gives us that $\nu_\chi(p_z)>0=\nu_\chi(1,z)$. For $z=b^{\pm1}_{j}$, $1<j\leq n$ (the case $j=1$ is trivial), we obtain $p_{z}$ satisfying Theorem \ref{geometriccrit} by using relation \ref{it:add5}. For $z=C^{\pm1}_{j,k}$, $1<j<k\leq n$, we have that $z$ and $b_{1}$ commute by relation~(\ref{it:puras8}) from Theorem~\ref{puras}, so we easily obtain such $p_z$.

Finally, for $z=a^{\pm 1}_{1}$, by relation~(\ref{it:puras5}) from Theorem~\ref{puras}, we have
\[b_{1}a^{-1}_{1}b^{-1}_{1}=\left\{\begin{aligned} a^{-1}_{1}(\prod^{n}_{j=2} C^{-1}_{1,j}C_{2,j}),&\quad&M=\mathbb T,\\
(\prod^{n}_{j=2} C_{1,j}C^{-1}_{2,j})a_{1},&\quad&M=\mathbb K,\end{aligned}\right.\]
 which gives the desired $\chi$-positive path $p_z$. Then $[\chi] \in \Sigma^{1}(P_{n}(M))$, as we wanted.

On case~\ref{casoiii}, we fix $t=b_n$ to use Theorem \ref{geometriccrit}. Notice that, by~(\ref{ordem}), the paths $p_i=(b_{n},b_{n-1}\cdots b_{i})=(b_{n},\beta_{n-1,i})$ are $\chi$-positive for all $2\leq i < n$. Thus, we obtain a path $p_{z}$ from $b_{n}$ to $zb_{n}$ in the Cayley graph by using relation~\ref{it:c1}
from Lemma~\ref{caminho} if $z=a^{\pm1}_{i}$ if $1\leq i\leq n$, relation~(\ref{it:puras6}) (resp. relation~(\ref{it:puras8})) from Theorem~\ref{puras} if $z=b^{\pm1}_{i}$ if $1\leq i <n$ (resp. $z=C^{\pm1}_{i,j}$, $1\leq i<j< n$). And if $z=C^{\pm1}_{i,n}$, for $1\leq i <n$, then by Proposition~\ref{prop:add}~\ref{it:add4}, the equality
\[\displaystyle b^{-1}_{n}C_{i,n}b_{n}=\left\{\begin{aligned}\prod^{n-i}_{j=1}b^{-1}_{n}C^{-1}_{n-j+1,n}C_{n-j,n}b_{n}&=\prod^{n-i}_{j=1}b_{n-j}C_{n-j,n}C^{-1}_{n-j+1,n}b^{-1}_{n-j},&\quad&M=\mathbb T,\\
\prod^{n-i}_{j=1}b^{-1}_{n}C^{-1}_{n-j+1,n}C_{n-j,n}b_{n}&=\prod^{n-i}_{j=1}b_{n-j}(C_{n-j,n}C^{-1}_{n-j+1,n})^{-1}b^{-1}_{n-j},&\quad&M=\mathbb K,\end{aligned}\right.
\]
provides us with a $\chi$-positive path, since $\chi(b_{n})+\chi(b_{j})>0$, for all $1\leq j\leq n-1$. Then $[\chi] \in \Sigma^{1}(P_{n}(M))$, as desired.

On cases~\ref{caso1} (resp.~\ref{caso2}), we can prove that $[\tilde{\chi}]\in \Sigma^{1}(P_{n}(\mathbb T))$, where $\tilde{\chi}(a_{\tau(i)})=\chi(a_{i})$ and $\tilde{\chi}(b_{\tau(i)})=\chi(b_{i})$, with $\tau\in S_{n}$ such that 
\begin{align*}
        \tilde{\chi}(a_{1})\leq \tilde{\chi}(a_{2})\leq\cdots\leq\tilde{\chi}(a_{n}),
\end{align*}
and by Corollary~\ref{cor:sigmapermtoro} it will follow that $[\chi] \in \Sigma^{1}(P_{n}(M))$ as well. Notice that $\tau(j)=1$ (resp. $\tau(j)=n$), and $\tilde{\chi}(a^{-1}_{1})+\tilde{\chi}(a^{-1}_{k})>0$ (resp. $\tilde{\chi}(a_{n})+\tilde{\chi}(a_{k})>0$), for all $1\leq k \leq n$. Now, this case is very similar to case~\ref{casoii} (resp.~\ref{casoiii}), providing us paths for $[\tilde{\chi}]$ by using $t=a^{-1}_{1}$ (resp. $t=a_{n}$).
On case~\ref{caso1}, the relations~(\ref{it:puras1}),~(\ref{it:puras5}) and~(\ref{it:puras8}) of Theorem~\ref{puras}, and also relations~\ref{it:add1} and~\ref{it:add3} of Proposition~\ref{prop:add} give us all paths satisfying Theorem~\ref{geometriccrit}. On case~\ref{caso2}, the paths satisfying Theorem~\ref{geometriccrit} are given by relations~(\ref{it:puras1}) and~(\ref{it:puras8}) of Theorem~\ref{puras}, by relation~\ref{it:c3} of Proposition~\ref{caminho}, and also by the following relation, which is a consequence of Proposition~\ref{prop:add}~\ref{it:add3}.
\[a^{-1}_{n}C_{i,n}a_{n}=\prod^{n-1}_{j=i}a^{-1}_{n}C_{j,n}C^{-1}_{j+1,n}a_{n}=\prod^{n-1}_{j=i}a_{j}C^{-1}_{j+1,n}C_{j,n}a^{-1}_{j}.\]

On case~\ref{casoiv}, we must look more carefully. If $n=3$, then $\chi(b_2)=0$. So, if $\chi(a_{2})=0$, we are on case~\ref{casoi}. Therefore, if $M=\mathbb K$ we have nothing else to analyse and, if $M=\mathbb T$, all the possibilities were analysed on cases~\ref{casoi},~\ref{caso1},~\ref{caso2} and on Lemma~\ref{BNS.P3}, which shows that $[\chi] \in \Sigma^1(P_n(M))$. If $n\geq 4$, first, we have to analyse whether the path
\begin{equation*}\label{path}
   \tilde{p}=(b_{n},b_{n-1}\cdots b_{3}\cdot b_{1}\cdot b_{n-1}\cdots b_{2})=(b_{n},\beta_{n-1,3}b_{1}\beta_{n-1,2})
\end{equation*}
is $\chi$-positive or not. The path $\tilde{p}$ is the concatenation of the paths $p_3=(b_n,\beta_{n-1,3})$, $p_1=(\beta_{n,3},b_1)$ and $p_2=(\beta_{n,3}b_1,\beta_{n-1,2})$. Now, it follows from~(\ref{ordem}) that $p_{3}$ is $\chi$-positive. Also, $\nu_\chi(p_1)=\min\{\chi(\beta_{n,3}),\chi(\beta_{n,3})+\chi(b_1)\}$, so $\nu_\chi(p_1)>0$ iff $\chi(\beta_{n,3})+\chi(b_1)>0$. In addition, from (\ref{somazero}) and the fact that $\chi(b_{n}b_{1})=0$ (we are on case \ref{casoiv}) we have that $\chi(\beta_{n-1,2})=0$, so $\nu_\chi(p_2)=\chi(\beta_{n,3})+\chi(b_1) + \nu_\chi(1,\beta_{n-1,2})=\chi(\beta_{n,3})+\chi(b_1)$. Since $\nu_\chi(\tilde{p})=\min\{\nu_\chi(p_3),\nu_\chi(p_1),\nu_\chi(p_2)\}$, it follows that $\nu_\chi(\tilde{p})>0$ iff $\chi(\beta_{n,3})+\chi(b_1)>0$. In particular, if $\nu_\chi(\tilde{p})\leq 0$ then $-\chi(b_2)=\chi(\beta_{n,3})+\chi(b_1)\leq 0$. Since $\chi(b_{2})\leq 0$, we conclude that $\chi(b_{2})=0$ and, consequently, $\chi(b_{k})=0$ for all $2\leq k\leq n-1$, by~(\ref{ordem}). We will use this information below.

If $\tilde{p}$ is $\chi$-positive, then fix $t=b_{n}$ and $\tilde{p}=p_{z}$ gives a path for $z=C^{\pm1}_{i,n}$ if $1\leq i< n$, by using Lemma~\ref{caminho}~\ref{it:c2}. For the other paths, we can use the fact that the paths are also $\chi$-positive for all $2\leq i < n$, and we choose the same paths $p_{z}$ from $b_{n}$ to $zb_{n}$ as in case~\ref{casoiii} above, if $z\in\left\{a^{\pm1}_{i},b^{\pm}_{i}, C^{\pm1}_{j,k}\,|\,1\leq i\leq n,\ 1\leq j<k<n\right\}$. Then $[\chi] \in \Sigma^{1}(P_{n}(M))$, as desired. If $\tilde{p}$ is not $\chi$-positive, then $\chi(b_{j})=0$, for all $2\leq j \leq n-1$, as we showed above. There are two possibilities. The first one is $\chi(a_{j})=0$ for some $2\leq j\leq n-1$, in which case it follows by~\ref{casoi} that $[\chi]\in\Sigma^{1}(P_{n}(M))$. The second possibility is if $\chi(a_{j})\neq 0$ for all $2\leq j\leq n-1$ (this case can only occur if $M=\mathbb T$). Then, by Corollary~\ref{cor:sigmapermtoro}, we can analyse the element $[\tilde{\chi}]$ insted of $[\chi]$, where $\tilde{\chi}(a_{\tau(j)})=\chi(a_{j})$ and $\tilde{\chi}(b_{\tau(j)})=\chi(b_{j})$, with $\tau\in S_{n}$ such that 
\begin{align*}
        \tilde{\chi}(a_{1})\leq \tilde{\chi}(a_{2})\leq\cdots\leq\tilde{\chi}(a_{n}),
\end{align*}
and the result for $[\tilde{\chi}]$ implies the same for $[\chi]$. We have $|\tilde{\chi}(a_1)|=|\tilde{\chi}(a_n)|$; otherwise, we would be either on case \ref{caso1} or \ref{caso2}. So, $\tilde{\chi}(a_1^{-1})=\tilde{\chi}(a_n)>0$. Now, consider the path
\begin{align*}
    \tilde{q}=(a_{n},a_{n-1}\cdots a_{3}\cdot a_{1}\cdot a_{n-1}\cdots a_{2})=(a_{n},\alpha_{n-1,3}a_{1}\alpha_{n-1,2})
\end{align*}
This path is analogous to the path $\tilde{p}$ and we proceed in a similar manner. If $\tilde{q}$ is $\tilde{\chi}$-positive, then fix $t=a_{n}$ and $\tilde{q}=p_{z}$ gives a path for $z=C^{\pm1}_{i,n}$ if $1\leq i< n$, by using Lemma~\ref{caminho}~\ref{it:c4}. For the other paths, we can use the fact that the paths
$q_i=(a_{n},a_{n-1}\cdots a_{i})=(a_{n},\alpha_{n-1,i})$ 
are also $\tilde{\chi}$-positive for all $2\leq i < n$, and we choose the path $p_{z}$ satisfying Theorem~\ref{geometriccrit} given by relations~(\ref{it:puras1}) for $z=a^{\pm}_{i}$, $1\leq i<n$ (resp. relation~(\ref{it:puras8}) for $z=C^{\pm}_{i,j}$, $1\leq i<j<n$) of Theorem~\ref{puras}, and by relation~\ref{it:c3} of Proposition~\ref{caminho} for $z=b^{\pm}_{i}$, $1\leq i\leq n$; therefore, $[\tilde{\chi}]\in \Sigma^{1}(P_{n}(M))$. 
If $\tilde{q}$ is not $\tilde{\chi}$-positive, then we can conclude that $\tilde{\chi}(a_{k})=0$ for all $2\leq k \leq n-1$ by the same arguments used before for the path $\tilde{p}$. Thus, at most, we have two elements, namely $a_{1}$ and $a_{n}$ that do not vanish $\tilde{\chi}$.  But, since $\tilde{\chi}$ is only a permutation of coordinates of $\chi$ and $\chi(a_{j})\neq 0$ for all $2\leq j\leq n-1$, the number $l$ of nonvanishing coordinates $\tilde{\chi}(a_j)$ is at least $n-2$. So, $n-2\leq l \leq 2$, which implies $n\leq 4$. Since we are on the case $n \geq 4$, we conclude that $n=4$, and the result follows by Lemma~\ref{BNS.P4}. Finally, it is easy to see that the circles defining $\Sigma^1(P_n(\mathbb T))^c$ are pairwise disjoint. This completes our proof.
\end{proof}

\section{Applications}\label{sec:appl}

The computation of the BNS invariant of a group is relevant on its own, but it may also bring some information about the finite generation of normal subgroups with abelian quotient \cite[Corollary B1.8]{Strebel} and on twisted conjugacy \cite{DaciDess}. We finish with some immediate consequences of our work. First, we see that the commutator subgroups of some pure braid groups are not finitely generated. It is worth noticing that the result below can be also known via other methods (see \cite{GG, GG3,GP}).

\begin{cor}\label{cor:g'naoef.g.}
Let $M=\mathbb T$, or $M=\mathbb K$, or $M={\mathbb S}^2$, and $n \geq 1$. Then the commutator subgroup $\gamma_2(P_n(M))$ is finitely generated if and only if $(M,n) \in \{(\mathbb T,1),(\mathbb K,1),(\mathbb S^2,1),(\mathbb S^2,2),(\mathbb S^2,3)\}$.
\end{cor}

\begin{proof}
As we observed in the beggining of Section \ref{sec:computation}, the groups $P_1(\mathbb T) \simeq \mathbb{Z} \times \mathbb{Z}$ and $P_1(\mathbb K) \simeq \mathbb{Z} \rtimes \mathbb{Z}$ are well known to have full BNS invariants and, therefore, finitely generated commutator subgroups \cite[Theorem A4.1]{Strebel}. By \cite{FVB}, the groups $P_n(\mathbb S^2)$ for $n=1,2,3$ are finite, so must have finitely generated commutator subgroups. For the other groups, by Theorems \ref{BNS-Klein} and \ref{th:PnS2}, $\Sigma^1(P_n(M))^c$ is not empty. The corollary then follows directly from Theorem A4.1 of \cite{Strebel}.
\end{proof}

Another motivation to study the BNS invariant for surface braid groups is that $\Sigma$-theory can be used on the investigation of the $R_\infty$ property for groups~\cite{DaciDess}. Recall that an automorphism $\varphi \in Aut(G)$ induces an equivalence relation called \emph{twisted conjugacy} on $G$, given by $x \sim_\varphi y \ \iff\ \exists\ z \in G:\ zx\varphi(z)^{-1}=y$. The number of equivalence classes (or Reidemeister classes) is denoted by $R(\varphi)$, and we say $G$ has property $R_\infty$ if $R(\varphi)$ is infinite for every $\varphi \in Aut(G)$. Twisted conjugacy has many connections with other areas of mathematics, including topological fixed point theory~\cite{Jiang}. We refer to the introduction of the paper \cite{DG} for a discussion of the historical context and development of $R_\infty$. The first proof of $R_\infty$ for the pure Artin braid groups $P_n$, $n \geq 3$, was published in 2021 \cite{DGO} and, in 2022, an alternative proof was obtained in \cite{CS}. This raises the question of whether pure braid groups of other closed surfaces $M$ have property $R_\infty$. This is an open problem that has been currently investigated and to which we give below a small contribution for the case of the Klein bottle, by using the $\Sigma^1$ invariant.

\begin{proof}[Proof of Corollary~\ref{cor:twistedpnklein}]
This is a straightforward consequence of Theorem \ref{BNS-Klein} and the proof of Corollary 3.4 of \cite{DaciDess}. In fact, since the complement of $\Sigma^1$ is invariant under automorphisms, there is a natural action by bijections $Aut(P_n(\mathbb K)) \curvearrowright \Sigma^1(P_n(\mathbb K))^c$, which induces a homomorphism $Aut(P_n(\mathbb K)) \to S_{k}$, where $k=2 \binom{n}{2}$. Let $H$ be the kernel of this homomorphism. Since the characters of $\Sigma^1(P_n(\mathbb K))^c$ are all discrete, $R(\varphi)$ is infinite for every $\varphi \in H$ (see \cite[Corollary 3.4]{DaciDess}). The fact $|Aut(P_n(\mathbb K)):H|\leq |S_k|=\left(2 \binom{n}{2}\right)!$ is a direct consequence of the First Isomorphism Theorem.
\end{proof}

Corollary \ref{cor:twistedpnklein} suggests a possibility of property $R_\infty$ to hold for the pure braid groups $P_n(\mathbb K)$. One reason is that a similar situation occurred in the year of 2019 for the Artin pure braid groups $P_n$: the second author, together with the authors of \cite{KW}, had the knowledge of a subgroup $H$ of index 2 of $Aut(P_n)$ such that $R(\varphi)$ was infinite for every $\varphi \in H$. One year later, as we said, the first proof of $R_\infty$ for $P_n$ was published in \cite{DGO}.

\end{document}